\documentclass[reqno, 11pt, a4paper, oneside]{amsart}
\usepackage{amsmath}
\usepackage{amsfonts}
\usepackage{amssymb}
\usepackage[ansinew]{inputenc}
\usepackage{graphicx}
\usepackage{amsbsy}
\usepackage{fancyhdr}
\usepackage{yfonts}
\usepackage{hyperref}
\usepackage{enumerate}

\numberwithin{equation}{section}
\newtheorem{teo}{Theorem}[section]
\newtheorem{prop}[teo]{Proposition}
\newtheorem{lema}[teo]{Lemma}

\newtheorem{coro}[teo]{Corollary}
\theoremstyle{definition}
\newtheorem{defi}[teo]{Definition}

\theoremstyle{remark}

\newtheorem*{ack}{Acknowledgements}
\sffamily
\title{The polynomial method over varieties}
\author{Miguel N. Walsh}
\address{Departamento de Matemática e IMAS-CONICET, Facultad de Ciencias Exactas y Naturales, Universidad de Buenos Aires, 1428 Buenos Aires, Argentina}
\email{mwalsh@dm.uba.ar}

\begin{document}

\def\F{\mathbb{F}}
\def\Fqn{\mathbb{F}_q^n}
\def\Fq{\mathbb{F}_q}
\def\Fp{\mathbb{F}_p}
\def\Di{\mathbb{D}}
\def\E{\mathbb{E}}
\def\Z{\mathbb{Z}}
\def\Q{\mathbb{Q}}
\def\C{\mathbb{C}}
\def\R{\mathbb{R}}
\def\N{\mathbb{N}}
\def\P{\mathbb{P}}
\def\T{\mathbb{T}}
\def\modp{\, (\text{mod }p)}
\def\modN{\, (\text{mod }N)}
\def\modq{\, (\text{mod }q)}
\def\modone{\, (\text{mod }1)}
\def\Zn{\mathbb{Z}/N \mathbb{Z}}
\def\Zp{\mathbb{Z}/p \mathbb{Z}}
\def\Zan{a^{-n}\mathbb{Z}/ \mathbb{Z}}
\def\Zal{a^{-l} \Z / \Z}
\def\Pr{\text{Pr}}
\def\leftsize{\left| \left\{}
\def\rightsize{\right\} \right|}

\begin{abstract}
We establish sharp estimates that adapt the polynomial method to arbitrary varieties. These include a partitioning theorem, estimates on polynomials vanishing on fixed sets and bounds for the number of connected components of real algebraic varieties. As a first application, we provide a general incidence estimate that is tight in its dependence on the size, degree and dimension of the varieties involved.
\end{abstract}

\maketitle

\tableofcontents

\section{Introduction}

The polynomial method is a powerful tool for establishing results in a wide variety of areas by means of the construction of an adequate polynomial \cite{G,T0,W}. By its very nature, many of the arguments involving this method require us to inductively study what happens inside the varieties produced by this polynomial. Because of this, one may be lead to the study of how the polynomial method adapts to general algebraic varieties not only if this is the natural setting of our problem, but even if the question originally takes place in a fixed variety like $\R^n$.

The purpose of this article is to provide sharp versions over general algebraic varieties of the basic tools of the polynomial method. These include a polynomial partitioning theorem, a form of Siegel's lemma for general algebraic sets, estimates on the irreducible components produced by families of polynomials vanishing on fixed sets, a bound on the number of connected components of real algebraic varieties and a bound for how many components of the complement of a polynomial can be intersected by an algebraic variety of given degree and dimension. 

It should be noted that some of these results involve questions regarding real algebraic varieties that certainly have a history of their own right, independently of their connection to the polynomial method. That being said, we are especially interested in how these estimates fit together in the context of this method and in particular, we believe these results may prove to be useful to extend its applications. 

As a first example of how these results can be applied, we provide a general incidence estimate for hypersurfaces of real varieties of arbitrary dimension, that is expected to be sharp in its dependence on the size, degree and dimension of the objects involved. This improves on the best known bounds even in the case of $\R^n$. In a separate article, we will further expand on the tools developed in this article to show how these ideas can be applied, both over $\R$ and over arbitrary fields, to obtain incidence estimates that are sensitive to how the elements being studied concentrate on varieties of smaller codimension \cite{W6}.

We now provide a more detailed description of our results. Let us begin with the polynomial partitioning theorem. Upon trying to adapt Dvir's use of the polynomial method in his solution of the Kakeya problem over finite fields \cite{D}, Guth and Katz \cite{GK} applied the idea of partitioning a set of points $S \subseteq \R^n$ by means of a polynomial $P$ of adequate degree, in such a way that each connected component of $\R^n \setminus Z(P)$ contains few elements of $S$. This idea has lead to a remarkable set of results \cite{G0,G3,GK}, leading also to some variations of this estimate being established \cite{BS,FPSSZ,G1,MS}. 

Of particular interest to us is a result of Basu and Sombra \cite{BS} that shows that, for points lying in a variety $V$ of codimension at most $2$, stronger partitionings can be produced as the degree of $V$ gets larger, and conjectured the same holds for varieties of arbitrary dimension \cite[Conjecture 3.4]{BS}. Our first result, Theorem \ref{1}, answers this affirmatively. Given an irreducible variety $V \subseteq \C^n$, we will write $\delta(V)$ for the minimal integer such that $V$ is an irreducible component of $Z(f_1,\ldots,f_r)$ for some polynomials $f_i$ of degree at most $\delta(V)$ (we refer to Section \ref{S2} for some further notation and definitions, including the asymptotic notation used below). We have the following result.

\begin{teo}
\label{1}
Let $V \subseteq \C^n$ be an irreducible variety of dimension $d$ and $S$ a finite set of points inside of $V(\R)$. Then, given any integer $M \ge \delta(V)$, we can find some polynomial $g \in \R[x_1,\ldots,x_n]$ of degree $O_{n}(M)$ such that $g$ does not vanish identically on $V$ and each connected component of $\R^n \setminus Z(P)$ contains 
$$\lesssim_{n} \frac{|S|}{M^d \deg(V)},$$
elements of $S$.
\end{teo}

We deduce this from a more general result, Theorem \ref{30par}, that is slightly harder to state but gives a corresponding estimate without requiring the restriction $M \ge \delta(V)$ to be imposed. We should remark that the more general form of Theorem \ref{30par} ends up being crucial in the subsequent applications.

To understand why an estimate of this kind could be relevant it helps to first notice that an estimate of the form $\lesssim_{n,\deg(V)} |S| M^{-d}$ for the number of elements in each connected component could be deduced from the partitioning result for $\R^d$. Theorem \ref{1} does not just make the dependence on $\deg(V)$ explicit but also shows that it actually improves as the degree of $V$ gets larger. To emphasize that estimates of this kind may in fact improve as the degree of the underlying variety gets larger is in fact one of the points of this article. 

The reason this can prove quite useful in practice is because, in many circumstances, an optimal application of the polynomial method requires us to construct a polynomial of large degree. When we subsequently want to deal with the algebraic set that this produces, it forces us to study varieties of high-degree and the general tools we have at our disposal may become much weaker in this context, making the problem unmanageable. This is the reason why many applications of the polynomial method proceed by truncating what would be the optimal polynomial that the problem would require us to construct, so that we only produce manageable low-degree varieties (e.g. \cite{BP,FPSSZ,GZ,HB,SoT}). Unfortunately, even in the contexts where this is possible, it tends to come at the cost of producing weaker results. By showing that some estimates may actually become stronger when the degree of the variety is large, results like Theorem \ref{1} open the door to countering those parts of the method that become less effective and thus make possible the study of high-degree varieties and the corresponding application of the method.

Our second result is in the same spirit. It deals with the basic problem of finding a polynomial vanishing on a given algebraic set while preventing some other variety from belonging to the resulting zero set.

\begin{teo}
\label{2}
Let $0 \le l < d \le n$ be integers. Let $V \subseteq \C^n$ be a $d$-dimensional algebraic set in $\C^n$ and $\tau_l > 0$ a real number. Let $T$ be an $l$-dimensional algebraic set of $\C^n$ with $\deg(T) \ge \tau_l \delta(V)^{d-l}\deg(V)$. Then, there exists some polynomial $P \in \C[x_1,\ldots,x_n]$ of degree at most
\begin{equation}
\label{Rbound2}
 \lesssim_{n,\tau_l} \left( \frac{\deg(T)}{\deg(V)} \right)^{\frac{1}{d-l}},
 \end{equation}
vanishing at all elements of $T$ without vanishing identically on $V$.
\end{teo}

As with Theorem \ref{1}, this result is a particular instance of more general results established in Section \ref{S4} that provide corresponding estimates without any restriction on the degrees of the varieties and this becomes important in applications. Because of the simple nature of their statements, we believe some form of the results of Section \ref{S4} may already be present in the literature in some way or another (we particularly refer the reader to \cite{CP} and our use of Theorem \ref{8tool}). On the other hand, also by their simple nature and the reasons previously discussed, we believe these results are likely to be useful tools to have in this generality when applying the polynomial method in different contexts. In particular, it should be remarked that these results hold over any algebraically closed field and will be used in a separate article to obtain some general incidence estimates over arbitrary fields. 

Theorem \ref{2} can be used to establish a number of useful estimates. For example, it yields the following asymptotic converse of Bezout's inequality (see Theorem \ref{clases} below), refining a result of Chardin and Philippon \cite[Theorem B]{CP}.

\begin{teo}
\label{3}
Let $V \subseteq \C^n$ be an irreducible variety of dimension $d$. Then $V$ is an irreducible component of $Z(f_1,\ldots,f_{n-d})$ for some polynomials $f_i$ with $\prod_{i=1}^{n-d} \deg(f_i) \lesssim_n \deg(V)$.
\end{teo}

In Section \ref{S5} and Section \ref{S6} we build upon Theorem \ref{2} and its variants to establish some relevant generalisations of Theorem \ref{3} that will be needed later. Section \ref{S6} is in fact an important part of this article. There we show how given a variety $V$ of dimension $d$, upon allowing some further flexibility in the upper bounds for the degrees of the polynomials $f_i$ that appear in Theorem \ref{3}, we can obtain an important amount of control on the higher-dimensional components of $Z(f_1,\ldots,f_{n-d})$. In order to do this, we introduce the concepts of envelopes and full covers associated with an irreducible variety $V \subseteq \C^n$ and establish the pertinent estimates for these sets. These ideas allow us to accurately model $V$ by means of an algebraic set defined by polynomials of the smallest possible degree we could expect.

We apply these results to the study of the $0$-th Betti number of a real variety $V \subseteq \R^n$. The classical work of Milnor \cite{M} and Thom \cite{T} (see also \cite{OP}) shows that if $V \subseteq \R^n$ is a real variety defined by polynomials of degree at most $D$, then the number $b_0(V)$ of connected components of $V$ is $O_n(D^n)$. Unfortunately, it turns out this estimate is not suitable for the study of the high-degree varieties arising when applying the polynomial method in high-dimensional problems. 

Let us briefly pause to discuss this. It is clear that the larger the degree of $V$ is, the larger the number of connected components we would expect it to have, and so this constitutes an example of the kind of estimates that will necessarily become less effective when studying high-degree varieties. We have previously discussed how estimates like Theorem \ref{1} and Theorem \ref{2} are intended to be the tools that allow us to counter these losses. However, the saving these results produce are only proportional to the algebraic degree $\deg(V)$ of $V$, while an estimate of the kind provided by Milnor and Thom produces losses that are proportional to a power of the largest degree of the polynomials needed to define $V$ and of course, this product can be much larger than $\deg(V)$. 

A result that is particularly important to our work is due to Barone and Basu \cite{BB}. They show that given polynomials $f_1,\ldots,f_{n-d} \in \R[x_1,\ldots,x_n]$, with $\deg(f_1) \le \ldots \le \deg(f_{n-d})$, and provided the real dimension of the algebraic set $Z(f_1,\ldots,f_i)$ is at most $n-i$ for every $1 \le i \le n-d$, the number of connected components of $Z(f_1,\ldots,f_{n-d})(\R)$ is bounded by
 $$ \lesssim_n \left(\prod_{i=1}^{n-d} \deg(f_i) \right) \deg(f_{n-d})^d,$$
see  \cite[Theorem 4]{BB}. By Theorem \ref{3} this is closer to what we would like to have and the fact that our variety of interest may only be one of many irreducible components of $Z(f_1,\ldots,f_{n-d})$ would not be detrimental in the applications. On the other hand, the dimensionality assumptions this result places on $Z(f_1,\ldots,f_i)$ are indeed very restrictive in practice and for a general variety stop us from obtaining any improvement over Milnor and Thom's bound.

Our next result addresses this problem. In Section \ref{S7} we show how the results we have established on envelopes and full covers allow us to model any variety accurately enough as to be able to carry the arguments of Barone and Basu without any additional requirements being placed on the variety. This produces the following result.

\begin{teo}
\label{4}
Let $V \subseteq \C^n$ be an irreducible variety of dimension $d$. Then there exists some algebraic set $X \subseteq \C^n$, having $V$ as an irreducible component, with $\deg(X) \lesssim_n \deg(V)$ and such that the number $b_0(X(\R))$ of connected components of $X(\R)$ satisfies
$$b_0(X(\R)) \lesssim_n \delta(V)^d \deg(V).$$
\end{teo}

What this result accomplishes is, at the cost of replacing $V$ by an algebraic set of essentially the same degree that retains $V$ as an irreducible component, producing a bound that is substantially stronger in general than that of Milnor and Thom.

We apply similar ideas to study the closely related problem of how many connected components of the complement $\R^n \setminus Z(P)$ of a polynomial $P$ can be intersected by a variety $V$ of given degree and dimension. Remember that the idea of the partitioning theorem is, given a set $S$, to find an adequate polynomial $P$ such that each connected component of $\R^n \setminus Z(P)$ contains few elements of $S$. If we can show that a variety $V$ only intersects a few of these components, then both facts combined would limit the amount of interaction $S$ and $V$ can have. The discussion preceding Theorem \ref{4} regarding the nature of previously known bounds applies to this problem as well. Our next result provides the desired sharp dependence on the degree of $V$ for a general variety.

\begin{teo}
\label{5}
Let $V \subseteq \C^n$ be an irreducible variety of dimension $d$ and $P \in \R[x_1,\ldots,x_n]$. Then $V(\R)$ intersects $\lesssim_n \deg(V) \deg(P)^d$ connected components of $\R^n \setminus Z(P)$.
\end{teo}

Since the original work of Milnor and Thom results like these have found a large number of applications (see \cite{BPR} for a general survey). This makes it likely that Theorem \ref{4} and Theorem \ref{5} will prove to be useful beyond the scope of the polynomial method. Let us also notice that these results constitute a variant of \cite[Conjecture 2.10]{BS} and it is likely that the same methods can be adapted to yield that estimate.

We now provide some consequences of our results in incidence geometry. Let $V \subseteq \R^n$ be an irreducible variety, $S \subseteq V$ a set of points and $T$ a family of subvarieties of $V$. Incidence geometry is interested in the estimating the number of incidences between $S$ and $T$, given by
$$ \mathcal{I}(S,T) = | \left\{ (s,t) \in S \times T : s \in t \right\}|.$$
There is a large body of work estimating this quantity for specific families of varieties. Notice, however, that a non-trivial estimate cannot be obtained in full generality. While it is clear for example that we cannot have a large number of points, and a large number of lines, with all the lines intersecting all the points, the same is not true for general varieties. Consider for instance the case in which $T$ is a set of planes. Then, if these planes are chosen so that they all contain a given line $L$, placing all points of $S$ inside of $L$ we see that we cannot improve upon the trivial bound $\mathcal{I}(S,T) = |S| |T|$. 

Nevertheless, just requiring that this particular degeneracy fails to occur, a richer theory emerges. We say $S$ is $(k,b)$-free with respect to $T$, for a certain pair of integers $k,b \ge 1$, if we cannot find $k$ elements from $S$ all of them lying inside $b$ different elements from $T$ (see Definition \ref{free}). Under this assumption, we will prove a sharp incidence estimate for hypersurfaces over general varieties. 

For $(k,d) \neq (1,1)$, let us write
$$ \alpha_k(d) = \frac{k(d-1)}{dk-1}, \, \beta_k(d) = \frac{d(k-1)}{dk-1}.$$
We set $\alpha_1(1)=0$ and $\beta_1(1)=1$. Also, let 
$$\tau_d(b,k) = b^{1-\beta_k(d)} k^{1-\alpha_k(d)}.$$
We will establish the following result.

\begin{teo}
\label{6}
Let $V \subseteq \C^n$ be an irreducible variety of dimension $d$. Let $T$ be a set of hypersurfaces of $\C^n$ and $S \subseteq V(\R)$ a set of points that is $(k,b)$-free with respect to $T$. Then $\mathcal{I}(S,T)$ is bounded by 
$$ c_1 |S|^{\alpha_k(d)}\deg(T)^{\beta_k(d)} \deg(V)^{1-\alpha_k(d)} + k \deg(T) \deg(V) + (b-1) |S|,$$
with $c_1 \lesssim_n \tau_d(b,k)$.
\end{teo}

There are some aspects of this bound that warrant some discussion. We begin by singling out the particular case in which $V=\C^n$, all elements of $T$ have degree $O_n(1)$, $b=O_n(1)$ and where we ignore the precise dependence of the constants.

\begin{coro}
\label{7}
Let $T$ be a set of hypersurfaces of $\R^n$ of degree $O_n(1)$ and $S$ a set of points that is $(k,O_n(1))$-free with respect to $T$. Then 
$$ \mathcal{I}(S,T) \lesssim_{n,k} |S|^{\alpha_k(n)}|T|^{\beta_k(n)} + |T|+ |S|.$$
\end{coro}

Corollary \ref{7} answers a conjecture of Elekes and Szabó \cite[Section 2]{ES} and a conjecture of Basu and Sombra \cite[Conjecture 4.1]{BS}. It improves on an estimate obtained by Fox, Pach, Sheffer, Suk and Zahl \cite{FPSSZ}. When $n=2$, it recovers the Szemerédi-Trotter theorem when $T$ is a set of lines \cite{ST} and the Pach-Sharir theorem when $T$ is a set of algebraic curves \cite{PS}. The case $n=3$ recovers the estimate given by Zahl in \cite{Z}, while the case $n=4$ gives that of \cite{BS}. It also subsumes further results in \cite{BK,CEGSW,ES,KMSS,LSZ}.

However, Theorem \ref{6} is a substantially stronger estimate than Corollary \ref{7}. In particular, it provides a very explicit dependence on the degrees of the varieties involved and this aspect of the bound seems to be new in problems of this kind. Furthermore, we will show in Section \ref{83} how to conditionally construct examples realising the main term of the bound, showing that we expect this estimate to be tight in general. Theorem \ref{6} is also very explicit in its dependence of $b$ and $k$ and this makes it significant even when $b$ and $k$ are not constrained to be uniformly bounded as the size of $S$ and $\deg(T)$ grow, as it is usually the case in the literature.

We have thus obtained a result that is effective on the degrees of both $T$ and $V$ and there is a particular aspect of this dependence that is worth discussing. Notice that we can see this as a problem involving a subvariety of $V$ of codimension one with respect to $V$ and degree $\deg(V) \deg(T)$. From this point of view, it may seem surprising that $\deg(T)$ and $\deg(V)$ appear with different exponents in the main term. In fact, one can check that $1-\alpha_k(d)=\beta_k(d)/d$ and so this exponent is substantially smaller than that of the factor $\deg(T)$ in general.

This is part of a broader phenomenon along the lines that have been discussed so far in the article. This has to do with the fact that incidence estimates actually improve as the degree of the ambient variety $V$ gets larger and this turns out to happen in complete generality. The dependence on the degrees in the main term of Theorem \ref{6} should be seen as a factor of the form $(\deg(T)\deg(V))^{\beta_k(d)}$ relating to the varieties whose incidences with $S$ we are studying, divided by an additional power of $\deg(V)$ that is a saving factor coming from the high-degree nature of the ambient variety $V$. This is made explicit in a separate article \cite{W6} where we allow the varieties $T$ being studied to have low degree inside of $V$. 

We finish this introduction with one last remark regarding the term $(b-1)|S|$ in Theorem \ref{6}. This term is clearly necessary, since by requiring $S$ to be $(k,b)$-free with respect to $T$ we have not excluded the possibility that $T$ is a set of $b-1$ varieties, all of which contain a given subvariety $L$ where all the points of $S$ lie. A nice consequence of ensuring it appears in this sharp form is that it allows us to immediately deduce an estimate for the set of $r$-rich points $\mathcal{P}_r(T)$ of $T$ in the optimal range. This is the set of points that are incident to at least $r$ elements of $T$ and phrasing incidence estimates in terms of this set can turn out to be useful in applications. Using that $\frac{\beta_k(d)}{1-\alpha_k(d)}=d$ and $\frac{1-\beta_k(d)}{1-\alpha_k(d)}=\frac{d-1}{k-1}$, we deduce from Theorem \ref{6} the following result. 

\begin{coro}
\label{8}
Let $V \subseteq \C^n$ be an irreducible variety of dimension $d$ and $T$ a set of hypersurfaces of $\C^n$. Let $r \ge b$ and let $S$ be a maximal $(k,b)$-free subset of $\mathcal{P}_r(T) \cap V(\R)$. Then
$$ |S| \le k \deg(V) \left(  \frac{2 \deg(T)}{r-b+1} + \frac{c_2 b^{\frac{d-1}{k-1}} \deg(T)^d}{(r-b+1)^{\frac{1}{1-\alpha_k(d)}}} \right),$$
with $c_2= O_n(1)$.
\end{coro}

\begin{ack}
Part of this work was carried while the author was a Clay Research Fellow and a Fellow of Merton College at the University of Oxford. The author would like to thank Cosmin Pohoata and Martin Sombra for pointing out some typos in an earlier version of this manuscript, as well as an anonymous referee for some helpful suggestions.
\end{ack}

\section{Preliminaries}
\label{S2}

\subsection{Notation}

Given parameters $a_1,\ldots,a_r$ we shall use the asymptotic notations $X \lesssim_{a_1,\ldots,a_r} Y$ or $X=O_{a_1,\ldots,a_r}(Y)$ to mean that there exists some constant $C$ depending only on $a_1,\ldots,a_r$ such that $X \le C Y$. We write $X \sim_{a_1,\ldots,a_r} Y$ if $X \lesssim_{a_1,\ldots,a_r} Y \lesssim_{a_1,\ldots,a_r} X$. We shall write $|A|$ for the cardinality of a set $A$.

Given polynomials $f_1,\ldots,f_r \in \C[x_1,\ldots,x_n]$ we will write
$$ Z(f_1,\ldots,f_r) = \left\{ x \in \C^n : f_1(x) = \cdots = f_r(x) = 0 \right\},$$
for the corresponding zero set. For an irreducible variety $V \subseteq \C^n$ we write $\deg(V)$ for the degree of its projective closure with respect to the standard embedding of $\C^n$ into $\P^n$ and more generally, for an algebraic set $V$ with irreducible components $V_1,\ldots,V_s$ we write $\deg(V) = \sum_{i=1}^s \deg(V_i)$. By an algebraic set of dimension $d$ we mean an algebraic set all of whose irreducible components have dimension $d$. Even though such a definition is not standard, we find it more convenient for our purposes. We will write $I(V)$ for the ideal of an algebraic set $V$ and write $I_{\R}(V) \subseteq I(V)$ for its subset of real polynomials. We shall also write $V(\R)$ for the real points of $V$.

\subsection{Algebraic preliminaries}

We will be using the following form of Bezout's inequality  \cite[Theorem 7.7]{Hart}.

\begin{lema}[Bezout's inequality]
\label{Bezout}
Let $W \subseteq \C^n$ be an irreducible variety and $f_1,\ldots,f_s \in \C[x_1,\ldots,x_n]$ polynomials. Write $Z_1,\ldots,Z_r$ for the irreducible components of $Z(f_1,\ldots,f_s) \cap W$. Then
$$ \sum_{i=1}^r \deg(Z_i) \le \deg(W) \prod_{j=1}^s \deg(f_j).$$
\end{lema}

Given an irreducible variety $V \subseteq \C^n$, a particularly important role will be played by the following quantities.

\begin{defi}[Partial degree]
\label{delta}
For an irreducible variety $V \subseteq \C^n$ and every $1 \le i \le n-\dim(V)$ we let $\delta_i(V)$ stand for the minimal integer $\delta$ for which we can find a finite set of polynomials $g_1,\ldots,g_t \in \C[x_1,\ldots,x_n]$ of degree at most $\delta$ such that $V \subseteq Z(g_1,\ldots,g_t)$ and the highest dimension of an irreducible component of $Z(g_1,\ldots,g_t)$ containing $V$ is equal to $n-i$. We sometimes abbreviate $\delta_{n-\dim(V)}(V)$ as $\delta(V)$ and call this the partial degree of $V$. By convention we also write $\delta_0(V)=0$ and $\delta_i(V)=\infty$ for every $i>n-\dim(V)$.
\end{defi} 

Clearly, these quantities satisfy the following simple relation.

\begin{lema}
\label{dorder}
For every irreducible variety $V$ we have $\delta_i(V) \ge \delta_{i-1}(V)$ for every $i$.
\end{lema}

\begin{proof}
Let $g_1,\ldots,g_t$ be the polynomials in the definition of $\delta_i(V)$, so in particular the irreducible components of $Z(g_1,\ldots,g_t)$ containing $V$ of highest dimension have dimension $n-i$. Then, by Krull's Hauptidealsatz, there exists some subset $g_1,\ldots,g_r$ of these polynomials such that the irreducible components of $Z(g_1,\ldots,g_t)$ containing $V$ of the highest dimension have dimension $n-i+1$. This clearly implies that $\delta_i(V) \ge \delta_{i-1}(V)$.
\end{proof}

If $V$ is not necessarily irreducible we will use the following variant of the above definition of partial degree.

\begin{defi}
For an algebraic set $V \subseteq \C^n$ of dimension $d$ we write $\delta(V)$ for the smallest integer $\delta$ for which we can find polynomials $g_1,\ldots,g_t$ of degree at most $\delta$ such that every irreducible component of $V$ is also an irreducible component of $Z(g_1,\ldots,g_t)$.
\end{defi} 

Consider an algebraic set $V \subseteq \C^n$. To the ideal $I(V)$ of $V$ we can associate the affine Hilbert function 
$$H_{I(V)}(m) := \text{dim}_{\C} \left( \C[x_1,\ldots,x_n]_{\le m} / I(V)_{\le m} \right).$$
where $\C[x_1,\ldots,x_n]_{\le m}$ is the vector space of polynomials of degree at most $m$, while $I(V)_{\le m}$ are those members of $I(V)$ of degree at most $m$. Similarly, writing $I_{\R}(V)$ for the ideal of $\R[x_1,\ldots,x_n]$ consisting of the real polynomials of $I(V)$, we can consider the function
$$H_{I(V),{\R}}(m) := \text{dim}_{\R} \left( \R[x_1,\ldots,x_n]_{\le m} / I_{\R}(V)_{\le m} \right),$$
These functions are related by the following simple lemma.

\begin{lema}
\label{HI}
We have $H_{I(V),\R}(m) \ge H_{I(V)}(m)$ for every algebraic set $V \subseteq \C^n$ and every $m \ge 0$.
\end{lema}

\begin{proof}
Let $p_1,\ldots,p_r$ be a maximal subset of $\R[x_1,\ldots,x_n]_{\le m}$ projecting to linearly independent elements of $\R[x_1,\ldots,x_n]_{\le m} / I_{\R}(V)_{\le m}$, so in particular $H_{I(V),\R}(m) = r$. Let $q$ be any element of $\C[x_1,\ldots,x_n]_{\le m}$, which we can write as $q=q_1 + {\bf i} q_2$ with $q_1,q_2 \in \R[x_1,\ldots,x_n]_{\le m}$ and {\bf i} the imaginary unit. We know we can find coefficients $a_1,\ldots,a_r,b_1,\ldots,b_r \in \R$, such that both $q_1 - \sum_{i=1}^r a_i p_i$ and $q_2 - \sum_{i=1}^r b_i p_i$ belong to $I_{\R}(V)_{\le m}$. It follows that $q - \sum_{i=1}^r (a_i+{\bf i} b_i) p_i \in I(V)_{\le m}$ and therefore $H_{I(V)}(m) \le r = H_{I(V),\R}(m)$, as desired.
\end{proof}

Similarly to \cite{BS}, we will be using the following general lower bound for Hilbert functions.

\begin{teo}
\label{8tool}
Let $V \subseteq \C^n$ be an algebraic set of dimension $d$. Then, there exists some constant $c_0 \gtrsim_n 1$ such that, for every $m \ge 2 (n-d) \delta(V)$, we have the bound
$$ H_{I(V)}(m) \ge c_0 m^d \deg(V).$$
\end{teo}

\begin{proof}
This follows from \cite[Corollaire 3]{CP}.
\end{proof}

\section{Polynomial partitioning for varieties}

Given an irreducible variety $V \subseteq \C^n$ of dimension $d$ and an integer $0 \le i \le n-d$, we write 
$$\Delta_i(V) = \max \left\{ \frac{\deg(V)}{\delta_{i+1}(V) \cdots \delta_{n-d}(V)} , 1 \right\},$$
with the understanding that $\Delta_{n-d}(V)=\deg(V)$.
Notice that we have
\begin{equation}
\label{basic1}
\Delta_{i+1}(V) \le \delta_{i+1}(V)\Delta_{i}(V),
\end{equation}
with equality holding whenever $\Delta_{i}(V)>1$.

From now on we let $c_0$ be as in Theorem \ref{8tool}. We say a non-negative integer $i$ is admissible with respect to $V$ if $\delta_{i+1}(V) > 2 i \delta_i(V)$. Notice in particular that if $i$ is not admissible, then $\delta_i(V) \gtrsim_n \delta_{i+1}(V)$. Writing
$$ \mathcal{R}_i(V) = [c_1 \delta_i(V)^{n-i} \Delta_i(V), \frac{c_0}{2} \delta_{i+1}(V)^{n-i} \Delta_i (V) ],$$
for every $0 \le i \le n-d$ and an appropriate $c_1>0$, we then have the following simple observation.

\begin{lema}
\label{ref}
We can choose $c_1 \gtrsim_n 1$ such that, for every irreducible variety $V \subseteq \C^n$, every positive integer lies inside an interval of the form $\mathcal{R}_i(V)$ with $i$ admissible with respect to $V$.
\end{lema}

\begin{proof}
This follows from observing that for every $0 \le i \le n-d$, either $i$ is admissible with respect to $V$ or we have $\delta_s(V) \lesssim_n \delta_i(V)$ for the smallest admissible $s$ with $s > i$. Notice that $0$ and $n-d$ are always admissible, since $\delta_0(V)=0$ and $\delta_{n-d+1}(V) = \infty$. Choosing $c_1>0$ sufficiently small with respect to $n$ and using (\ref{basic1}), we can then guarantee that $\mathcal{R}_i(V) \cap \mathcal{R}_j(V) \neq 0$ if $i<j$ are both admissible with respect to $V$ and there is no other admissible integer in between them, yielding the result.
\end{proof}

Given an irreducible variety $V$ and a positve integer $M$, we write $i_V(M)$ for the smallest admissible $i$ such that $M^{n-i} \Delta_i(V) \in \mathcal{R}_i(V)$.
Clearly, we have
\begin{equation}
\label{ivmrange}
\delta_{i_V(M)}(V) \lesssim_{n} M \lesssim_n \delta_{i_V(M)+1}(V).
\end{equation}

We will prove the following more general form of Theorem \ref{1}.

\begin{teo}[Polynomial partitioning for varieties]
\label{30par}
Let $V \subseteq \C^n$ be an irreducible variety of dimension $d$ and $S$ a finite set of points inside of $V(\R)$. Then, given any integer $M \ge 1$, we can find some polynomial $g \in \R[x_1,\ldots,x_n] \setminus I(V)$ of degree $O_{n}(M)$ such that each connected component of $\R^n \setminus Z(P)$ contains 
$$\lesssim_{n} \frac{|S|}{M^{n-i_V(M)}\Delta_{i_V(M)}(V)}$$
elements of $S$.
\end{teo}

Notice that if $M \ge \delta(V)$ then, after multiplying $M$ by a sufficiently large $O_n(1)$ constant if necessary, we have $i_V(M)=n-d$. Since $\Delta_{n-d}(V)=\deg(V)$, we see that Theorem \ref{1} indeed follows from Theorem \ref{30par}.

We say a polynomial $g$ bisects a finite set $S \subseteq \R^n$ if we have 
$$ \left| \left\{ s \in S : g(s)>0 \right\} \right| \le |S|/2,$$
and
$$ \left| \left\{ s \in S : g(s)<0 \right\} \right| \le |S|/2.$$
Notice this does not exclude the possibility that a lot of the points actually lie on the zero set of $g$. Let us now state the well-known ham-sandwich theorem.

\begin{lema}[Ham-sandwich theorem]
\label{HS}
Let $S_1,\ldots, S_n$ be finite sets of points in $\R^n$. Then there exists a hyperplane bisecting every $S_i$.
\end{lema}

For the proof of Theorem \ref{30par} we will need to establish the following variant of the polynomial ham-sandwich theorem for sets of points lying inside a variety.

\begin{teo}
\label{30lem}
Let $V \subseteq \C^n$ be an irreducible variety of dimension $d$ and let $S_1,\ldots,S_k$ be finite subsets of $V(\R)$. Let $s$ be an admissible integer with respect to $V$ such that $k \in \mathcal{R}_s(V)$. Then there exists a real polynomial $g \notin I(V)$ of degree at most 
\begin{equation}
\label{Cnd}
\lesssim_{n} \left( \frac{k}{\Delta_s(V)} \right)^{\frac{1}{n-s}}
\end{equation}
that bisects every $S_i$.
\end{teo}

\begin{proof}
Since $k \in \mathcal{R}_s(V)$ and $s$ is admissible, we can find some positive integer 
\begin{equation}
\label{mrange}
2s \delta_s(V) \le m < \delta_{s+1}(V)
\end{equation}
 bounded above up to a constant by the expression (\ref{Cnd}) and satisfying the bound 
\begin{equation}
\label{lowerm}
c_0 m^{n-s}\Delta_s(V) > k,
\end{equation}
with $c_0$ as in Theorem \ref{8tool}. It will suffice to show that there exists some real polynomial $g \notin I(V)$ of degree at most $m$ bisecting every $S_i$.

Our first step will be to establish the following lemma, that we shall also need later. 

\begin{lema} 
\label{hk}
Let $V \subseteq \C^n$ be an irreducible variety of dimension $d$ and let $m$ be an integer satisfying (\ref{mrange}) and (\ref{lowerm}). Then $H_{I(V)}(m) > k$.
\end{lema}

\begin{proof}
To see this let $g_1,\ldots,g_t$ be as in the definition of $\delta_{s}(V)$. Let $V_1,\ldots,V_r$ be the nonempty set of all irreducible components of $Z(g_1,\ldots,g_t)$ of dimension $n-s$ containing $V$. We claim there is some $1 \le j \le r$ such that $I(V)_{\le m}=I(V_j)_{\le m}$. Clearly the inclusion $I(V_j)_{\le m} \subseteq I(V)_{\le m}$ always holds, so let us assume that for every $1 \le j \le r$ we can find some $h_j \in \C[x_1,\ldots,x_n]$ such that $h_j \in I(V)_{\le m} \setminus I(V_j)_{\le m}$. Then $Z(h_j) \cap V_j$ will be an algebraic set containing $V$ having all its irreducible components of dimension less than $n-s$. Hence $Z(g_1,\ldots,g_t,h_1,\ldots,h_r)$ will be an algebraic set containing $V$ such that its irreducible components that contain $V$ have dimension at most $n-s-1$. But all the polynomials $g_1,\ldots,g_t,h_1,\ldots,h_r$ have degree at most $m$. This implies that $\delta_{s+1}(V) \le m$, contradicting (\ref{mrange}). This proves our claim.

Let us then assume without loss of generality that $I(V)_{\le m}=I(V_1)_{\le m}$. Since $ m \ge 2s\delta_{s}(V) \ge 2(n-\dim(V_1)) \delta(V_1)$, we can apply Theorem \ref{8tool} to conclude that
\begin{equation}
\label{him}
 H_{I(V)}(m) = H_{I(V_1)}(m) \ge c_0 m^{n-s} \deg(V_1).
 \end{equation}
We now claim that
\begin{equation}
\label{degmax}
 \deg(V_1) \ge \Delta_{s}(V).
 \end{equation}
To see this, recall that $V_1$ is an $(n-s)$-dimensional irreducible variety containing $V$. By definition of $\delta_{s+1}(V)$ there must exist some polynomial $f_{s+1}$ of degree at most $\delta_{s+1}(V)$ vanishing on $V$ that cuts $V_1$ properly. In particular, there is some irreducible component $V_1^{(s+1)}$ of $Z(f_{s+1}) \cap V_1$ of dimension $n-s-1$ and degree at most $\delta_{s+1}(V) \deg(V_1)$ (by Lemma \ref{Bezout}) that contains $V$. Iterating this argument until we obtain an irreducible variety of dimension $d$ that contains $V$, and must therefore be equal to $V$, it follows that
$$ \deg(V) \le \deg(V_1) \delta_{s+1}(V) \cdots \delta_{n-d}(V).$$
This establishes (\ref{degmax}). Plugging this into (\ref{him}) and using (\ref{lowerm}), it follows that $H_{I(V)}(m) > k$, as desired.
\end{proof}

We now proceed to show that we can find a real polynomial $g \notin I(V)$ of degree at most $m$ bisecting every $S_i$. Let $1,p_1,\ldots, p_t$ be a basis of $\R[x_1,\ldots,x_n]_{\le m} / I_{\R}(V)_{\le m}$. Since we have established that $H_{I(V)}(m) > k$, it must be $t \ge k$ by Lemma \ref{HI}. To each $p_i$ we associate a representative $q_i \in \R[x_1,\ldots,x_n]_{\le m}$, that is to say, an element whose projection to $\R[x_1,\ldots,x_n]_{\le m} / I_{\R}(V)_{\le m}$ is equal to $p_i$. We consider the map $\phi: \R^n \rightarrow \R^t$ given by
$$ \phi(x)=\left( q_1(x), \ldots,q_t(x) \right).$$
If $x$ and $y$ are two different points inside of $V(\R)$, then we know there is some $1 \le i \le n$ such that the linear projection $\pi_i$ to the $i$th coordinate satisfies $\pi_i(x) \neq \pi_i(y)$. Since the elements of $I(V)$ vanish on both $x$ and $y$, and $1,p_1,\ldots,p_t$ is a basis for $\R[x_1,\ldots,x_n]_{\le m}/I(V)_{\le m}$, it follows that there is some linear combination of the $q_i$ that takes different values on $x$ and $y$. This implies that the map $\phi$ is injective on $V(\R)$. In particular, it is injective on each $S_i$. 

Consider now the sets $\phi(S_1),\ldots,\phi(S_k) \subseteq \R^t$. By Lemma \ref{HS} and the fact that $k \le t$, we know that there is some hyperplane in $\R^t$ bisecting each $\phi(S_i)$. This means that there are some coefficients $a_1,\ldots,a_{t+1} \in \R$, not all equal to zero, such that for every $S_i$ we have
$$ \left| \left\{ x \in S_i : a_1 q_1(x) + \ldots + a_t q_t(x) + a_{t+1} > 0 \right\} \right| \le |\phi(S_i)|/2 = |S_i|/2,$$
and
$$ \left| \left\{ x \in S_i : a_1 q_1(x) + \ldots + a_t q_t(x) + a_{t+1} < 0 \right\} \right| \le |\phi(S_i)|/2 = |S_i|/2.$$
Choosing $g = a_1 q_1 + \ldots + a_t q_t + a_{t+1}$, this concludes the proof of Theorem \ref{30lem}.
\end{proof}

Let us now turn to the proof of Theorem \ref{30par}. We will use the notation
$$ \Pi_i(V) = \delta_i(V)^{n-i} \Delta_i(V).$$

\begin{proof}[Proof of Theorem \ref{30par}] Let $S$ and $V$ be as in the statement of the theorem. By Theorem \ref{30lem} we can find a real polynomial $g_1 \notin I(V)$ of degree $O(1)$ bisecting $S$. Let us write $A_{1,1}$ for the points $x \in \R^n$ where $g_1(x)>0$ and $A_{1,2}$ for those with $g_1(x)<0$. Of course, $g_1$ vanishes on the remaining points. Clearly, each $A_{1,i}$ is the union of some open connected components of $\R^n \setminus Z(g_1)$. We write $S_{1,1}$ for those points of $S$ inside of $A_{1,1}$ and $S_{1,2}$ for those inside of $A_{1,2}$. We know both sets have size at most $|S|/2$. The points of $S$ that do not belong to any of these two sets must be contained inside of $Z(g_1)$.

We proceed recursively. Write
\begin{equation}
\label{r}
r=M^{n-i_V(M)}\Delta_{i_V(M)}(V),
\end{equation}
and suppose that given a positive integer $i \le \log_2 r$ we have constructed a real polynomial $g_{i-1} \notin I(V)$ and disjoint open sets $A_{i-1,1}, \ldots, A_{i-1,2^{i-1}}$, each of them being the union of some open connected components of $\R^n \setminus Z(g_{i-1})$. Suppose we have also guaranteed that, writing $S_{i-1,j}$ for those points of $S$ inside of $A_{i-1,j}$, then  $|S_{i-1,j}| \le |S|2^{-(i-1)}$ and that all points of $S$ outside of these sets lie inside of $Z(g_{i-1})$. We can now use Theorem \ref{30lem} to find a real polynomial $h_i \notin I(V)$ bisecting $S_{i-1,j}$ for every $1 \le j \le 2^{i-1}$. Notice that choosing $c_1 \gtrsim_n 1$ sufficiently small, we have by (\ref{basic1}) that the intervals
\begin{equation}
\label{rangoi}
 [c_1 \Pi_t(V), \frac{c_0}{2}\Pi_{t+1}(V)] \subseteq \mathcal{R}_t(V),
 \end{equation}
with $t$ admissible cover the positive integers, as it is easy to verify. If $t < i_V(M)$ is the smallest admissible integer with $2^{i-1}$ lying in an interval of the above form, we can use Theorem \ref{30lem} to bound the degree of $h_i$ by 
\begin{equation}
\label{hi}
\deg(h_i) \lesssim_n 2^{\frac{i-1}{n-t}} \Delta_{t}(V)^{-\frac{1}{n-t}}.
\end{equation}
Similarly, by definition of $i_V(M)$ and Theorem \ref{30lem}, we can take the remaining $h_i$ to satisfy
$$ \deg(h_i) \lesssim_n 2^{\frac{i-1}{n-i_V(M)}} \Delta_{i_V(M)}(V)^{-\frac{1}{n-i_V(M)}}.$$
 
Write $B_1$ for those points of $\R^n$ where $h_i$ is strictly positive and $B_2$ for those where it is strictly negative. These are open sets with boundary in $Z(h_i)$. Write $g_i=g_{i-1}h_i$ and notice in particular that $g_i \notin I(V)$ and $g_i$ is a real polynomial. For every $1 \le j \le 2^{i-1}$ define $A_{i,j}=A_{i-1,j} \cap B_1$ and $A_{i,2^{i-1}+j} = A_{i-1,j} \cap B_2$, so we are simply separating the elements of each $A_{i-1,j}$ according to the sign of $h_i$. The resulting sets are open sets which are the union of some open connected components of $\R^n \setminus Z(g_i)$. Since each $S_{i-1,j}$ is contained inside $A_{i-1,j}$ and by construction of $h_i$ has at most half its elements in $B_1$ and half of them in $B_2$, we conclude that writing $S_{i,j}$ for those elements of $S$ inside of $A_{i,j}$ we obtain a collection of $2^i$ sets, with each $S_{i,j}$ having at most $|S|2^{-i}$ elements of $S$. All elements of $S$ not lying inside of $S_{i,j}$ for any $j$ must lie inside of $Z(g_i)$. 

Repeating this process up to $i=\log_2 r$, we have found a real polynomial $g \notin I(V)$ and a partition of $\R^n \setminus Z(g) $ into sets $A_j$, $j=1,\ldots,r$, such that each $A_j$ is the union of some open connected components of $\R^n \setminus Z(g)$ and such that each $A_j$ contains at most $|S|/r$ elements of $S$. To finish the proof of Theorem \ref{30par} it thus only remains to show that $\deg(g) \lesssim_{n} M$. 

By our previous arguments we know that we can write $g=\prod_{i=1}^{\log_2 r} h_i$, where the polynomials $h_i$ have their degree bounded in the way described above. As a consequence, we have
\begin{equation*}
\begin{aligned}
 \sum_{i=1}^{1+\log_2 \frac{c_0}{2} \Pi_{i_V(M)}(V)} \deg(h_i) &\lesssim_{n} \sum_{t=0}^{i_V(M)-1} \sum_{i=\log_2 c_1 \Pi_{t}(V)}^{1+ \log_2 \frac{c_0}{2}  \Pi_{t+1}(V)} \Delta_t(V)^{-\frac{1}{n-t}}  2^{\frac{i-1}{n-t}} \\
 &\lesssim_{n} \sum_{t=0}^{i_V(M)-1}  \Delta_t(V)^{-\frac{1}{n-t}} \left( \delta_{t+1}(V)^{n-(t+1)} \Delta_{t+1}(V) \right)^{\frac{1}{n-t}} \\
 &\lesssim_{n} \sum_{t=0}^{i_V(M)-1}  \delta_{t+1}(V) \\
 &\lesssim_{n} M,
\end{aligned}
\end{equation*}
where we have used (\ref{basic1}) and (\ref{ivmrange}), and similarly
\begin{equation*}
\begin{aligned}
 \sum_{i=1+\log_2 \frac{c_0}{2} \Pi_{i_V(M)}(V)}^{\log_2 r} \deg(h_i) &\lesssim_{n} \Delta_{i_V(M)}(V)^{-\frac{1}{n-i_V(M)}}  \sum_{i=1+\log_2 \frac{c_0}{2}  \Pi_{i_V(M)}(V)}^{\log_2 r} 2^{\frac{i-1}{n-i_V(M)}} \\
 &\lesssim_{n}  \Delta_{i_V(M)}(V)^{-\frac{1}{n-i_V(M)}} r^{\frac{1}{n-i_V(M)}} \\
 &\lesssim_{n} M,
\end{aligned}
\end{equation*}
by (\ref{r}). This concludes the proof of Theorem \ref{30par}.
\end{proof}

\section{Siegel's Lemma for varieties}
\label{S4}

From Theorem \ref{30par} we can deduce the following estimate.

\begin{coro}
\label{case0}
Let $S$ be a finite set of points inside $V(\R)$ for some irreducible variety $V \subseteq \C^n$ of dimension $d$. Let $s$ be an admissible integer with $|S| \in \mathcal{R}_s(V)$. Then, there exists some polynomial $P$ of degree at most
$$ \lesssim_{n} \left( \frac{|S|}{\Delta_{s}(V)} \right)^{\frac{1}{n-s}},$$
vanishing on $S$ without vanishing identically on $V$.
\end{coro}

We can use a dimension counting argument to give a direct proof of this result, which we now formulate in a slightly more general form. For every $0 \le s \le n-d$, we shall extend the definition of the intervals $\mathcal{R}_s(V)$ to intervals of the form
$$ \mathcal{R}_{s,\tau}^l(V) = [\tau \delta_s(V)^{n-(s+l)} \Delta_{s}(V), \tau \delta_{s+1}(V)^{n-(s+l)} \Delta_s (V) ],$$
for every choice of real numbers $\tau > 0$, integers $0 \le l < n-s$ and irreducible varieties $V \subseteq \C^n$. 

As in the proof of Lemma \ref{ref}, the following observation follows immediately from the definition of the $\Delta_i$ and the fact that given a positive integer $s$, if $t$ is the smallest admissible integer with $s \le t$, then $\delta_s(V) \gtrsim_n \delta_t(V)$.

\begin{lema}
\label{coverN}
Let $V \subseteq \C^n$ be an irreducible variety. For any integers $l<d$ and $0 < \varepsilon < 1$, we can find $\varepsilon \lesssim_n \tau_1,\ldots,\tau_{n-d} \le \varepsilon$ such that $\R_{\ge 0}$ is covered by the sets $\mathcal{R}_{s,\tau_s}^l(V)$ with $s$ admissible.
 \end{lema}

We have the following variant of Corollary \ref{case0} that does not require the set of points $S$ to be real or to lie inside of $V$.

\begin{lema}
\label{complexcase}
Let $S$ be a finite subset of $\C^n$ and let $V$ be a $d$-dimensional irreducible variety $V \subseteq \C^n$. Let $\tau >0$ be sufficiently small with respect to $n$ and let $s$ be an admissible integer with $|S| \in \mathcal{R}_{s,\tau}^0(V)$. Then, there exists some polynomial $P$ of degree at most
\begin{equation}
\label{Cnd2}
\lesssim_{n,\tau} \left( \frac{|S|}{\Delta_{s}(V)} \right)^{\frac{1}{n-s}},
\end{equation}
vanishing on $S$ without vanishing identically on $V$.
\end{lema}

\begin{proof}
Since $|S| \in \mathcal{R}_{s,\tau}^0(V)$ and $s$ is admissible, as long as $\tau$ is chosen sufficiently small with respect to $n$, we can find some positive integer $2s \delta_s(V) \le m < \delta_{s+1}(V)$ bounded above by the expression (\ref{Cnd2}) and satisfying the bound 
\begin{equation}
\label{lowerm2}
c_0 m^{n-s}\Delta_s(V) > |S|,
\end{equation}
with $c_0 \gtrsim_n 1$ as in Theorem \ref{8tool}. By Lemma \ref{hk}, we know in particular that $H_I(m) > |S|$. This means that there exists a basis $p_1,\ldots,p_t$ of $\C[x_1,\ldots,x_n]_{\le m} \setminus I(V)_{\le m}$ with $t > |S|$. If for each $p_i$ we let $q_i$ be an element of $\C[x_1,\ldots,x_n]_{\le m}$ that projects to $p_i$, the fact that $t > |S|$ implies that there is some nonzero linear combination over $\C$ of $q_1,\ldots,q_t$ that vanishes on every element of $S$. Since by definition of the $q_i$ this linear combination does not vanish on $V$, the result follows.
\end{proof}

Similarly, we have the following variant that does not require $V$ to be irreducible. Of course, in this case, the assertion that the polynomial $P$ we construct does not vanish identically on $V$ does not prevent the possibility that the zero set of $P$ still contains many of the irreducible components of $V$. It only guarantees that it does not contain all of them.

\begin{lema}
\label{reduciblecase}
Let $V \subseteq \C$ be a $d$-dimensional algebraic set and $\tau > 0$ some real number. Let $S$ be a finite subset of $\C^n$ with $|S| \ge \tau \delta(V)^d \deg(V)$. Then, there exists some polynomial $P$ of degree at most
\begin{equation}
\label{Cnd3}
\lesssim_{n,\tau} \left( \frac{|S|}{\deg(V)} \right)^{1/d},
\end{equation}
vanishing on $S$ without vanishing identically on $V$.
\end{lema}

\begin{proof}
We begin by noticing that by our assumptions, we have
$$ \left( \frac{|S|}{\deg(V)} \right)^{1/d} \gtrsim_{\tau} \delta(V).$$
We see from this that by Theorem \ref{8tool}, we can find some integer $m$ bounded above by the expression (\ref{Cnd3}) and satisfying $H_{I(V)}(m) > |S|$. The proof then follows exactly as in Lemma \ref{complexcase}.
\end{proof}

We will need the following simple observation, that can be seen by considering a generic hyperplane intersecting $t$.

\begin{lema}
\label{Hconstruction}
Let $t \subseteq \C^n$ be an $l$-dimensional irreducible variety and let $H$ be a finite family of irreducible varieties of $\C^n$ of dimension $l-1$. Then there exists some irreducible variety $h \subseteq t$ of dimension $l-1$ with $\deg(h) \le \deg(t)$ and $h \notin H$.
\end{lema}

Recall that given a set $T$ of varieties of the same dimension, we write 
$$ \deg(T) = \sum_{t \in T} \deg(t).$$
We want to establish the following generalisation of Siegel's lemma to varieties of arbitrary dimension.

\begin{teo}
\label{GSiegel}
Let $0 \le l < d \le n$ be integers and $\tau_l >0$ a sufficiently small constant with respect to $n$. Let $T$ be a finite set of $l$-dimensional irreducible varieties in $\C^n$ and $V$ a $d$-dimensional irreducible variety in $\C^n$. Let $0 \le s \le n-d$ be an admissible integer with $\deg(T) \in \mathcal{R}_{s,\tau_l}^l(V)$. Then, there exists some polynomial $P \in \C[x_1,\ldots,x_n]$ of degree at most
\begin{equation}
\label{Rbound}
 \lesssim_{n,\tau_l} \left( \frac{\deg(T)}{\Delta_{s}(V)} \right)^{\frac{1}{n-(s+l)}},
 \end{equation}
vanishing at all elements of $T$ without vanishing identically on $V$.
\end{teo}

\begin{proof}
We proceed by induction on $l$, the case $l=0$ being handled by Lemma \ref{complexcase}. Let $R \ge 1$ be a parameter to be specified below. Given some $t \in T$ and some integer $r \ge 1$ to be specified soon, we apply Lemma \ref{Hconstruction} to find distinct irreducible subvarieties $h_1,\ldots,h_{r}$ of $t$ of dimension $l-1$ with $\deg(h_i) \le \deg(t)$ for every $i$. Because of this last bound, we may choose $r$ so that
$$ 2R \deg(t) \le \sum_{i=1}^r \deg(h_i) \le (2R+1) \deg(t).$$
We repeat this process for every element $t \in T$, applying Lemma \ref{Hconstruction} so that there is no overlap in the varieties $h_i$ obtained from different $t$, leading to a collection $\mathcal{H}$ of subvarieties with
$$ 2R \deg(T) \le \deg(\mathcal{H}) \le (2R+1) \deg(T).$$
Let
$$ R = B  \left( \frac{\deg(T)}{\Delta_{s}(V)} \right)^{\frac{1}{n-(s+l)}},$$
for some large $B \sim_{n,\tau_l} 1$ to be specified, so that
\begin{equation}
\label{degH}
 \deg(\mathcal{H}) = C_1 B \frac{\deg(T)^{1+\frac{1}{n-(s+l)}}}{\Delta_{s}(V)^{\frac{1}{n-(s+l)}}},
 \end{equation}
with $C_1 \sim 1$. Since $\deg(T) \in \mathcal{R}_{s,\tau_l}^l(V)$, it follows that
$$ \tau_l^{\frac{1}{n-(s+l)}} \delta_{s}(V) \le  \left( \frac{\deg(T)}{\Delta_{s}(V)} \right)^{\frac{1}{n-(s+l)}} \le \tau_l^{\frac{1}{n-(s+l)}} \delta_{s+1}(V),$$
and so in particular
\begin{equation}
\begin{aligned}
\label{newrange}
 (\tau_l^{1+\frac{1}{n-(s+l)}} C_1 B) \delta_{s}(V)^{n-(s+l-1)} \Delta_s(V) &\le \deg(\mathcal{H}) \\
 &\le (\tau_l^{1+\frac{1}{n-(s+l)}} C_1 B) \delta_{s+1}(V)^{n-(s+l-1)} \Delta_s(V).
 \end{aligned}
 \end{equation}
 Let now $\tau_{l-1} \sim_n 1$ be a sufficiently small fixed constant. We let $B_0 \sim_n 1$ be a sufficiently large constant with respect to $\tau_{l-1}$ and $n$. We then require that $\tau_l$ is sufficiently small as to satisfy 
 $$\tau_l^{1+\frac{1}{n-(s+l)}} C_1 B_0 \le \tau_{l-1}.$$
Finally, we choose $B \ge B_0$ such that 
 $$ \tau_l^{1+\frac{1}{n-(s+l)}} C_1 B = \tau_{l-1}.$$
 Notice that $B \lesssim_{n,\tau_l} 1$.
 
  We thus see from (\ref{newrange}) that $\deg(\mathcal{H}) \in R_{s,\tau_{l-1}}^{l-1}(V)$.  Since the components of $\mathcal{H}$ have dimension $l-1$, we are in a position to apply the induction hypothesis, provided $\tau_{l-1}$ was chosen sufficiently small. This gives us a polynomial $P$ of degree at most
 $$ \lesssim_{n,\tau_{l-1}} \left( \frac{\deg(\mathcal{H})}{\Delta_{s}(V)} \right)^{\frac{1}{n-(s+l-1)}},$$
 vanishing at all elements of $\mathcal{H}$, without vanishing identically on $V$. By (\ref{degH}) it follows that
 \begin{equation}
\label{quant}
  \deg(P) \lesssim_{n,\tau_{l-1}} B^{\frac{1}{n-(s+l-1)}} \left( \frac{\deg(T)}{\Delta_{s}(V)} \right)^{\frac{1}{n-(s+l)}}.
  \end{equation}
Since $B \ge B_0$ and $B_0$ can be taken to be a sufficiently large $O_n(1)$ quantity and since $n-(s+l-1) \ge 2$ and $\tau_{l-1} \lesssim_n 1$, we can take (\ref{quant}) to be strictly smaller than $R$. In other words, the polynomial $P$ can be taken to have degree less than $R$. If $P$ were to cut any element $t$ of $T$ properly then $Z(P) \cap t$ would have degree at most $\deg(t)R$. But this intersection would contain the corresponding components of $\mathcal{H} \cap t$, which were chosen to have degree at least $2R\deg(t)$, leading to a contradiction. It follows that $P$ must vanish at all elements of $T$ and this concludes the proof of Theorem \ref{GSiegel}.
\end{proof}

\begin{coro}
Let the hypothesis and notation be as in Theorem \ref{GSiegel}. If each $t \in T$ is the complexification of a real irreducible variety, then the polynomial $P$ can be taken to have real coefficients.
\end{coro}

\begin{proof}
Let $P=P_1 + i P_2$ be the polynomial provided by Theorem \ref{GSiegel}, with $P_1,P_2$ having real coefficients. Then both $P_1$ and $P_2$ have to vanish at every real point of every $t \in T$. But clearly both polynomials cannot be contained in $I(V)$, since then $P$ would also lie in $I(V)$. Assume without loss of generality that $P_1 \notin I(V)$. The result follows by renaming $P=P_1$ and observing that, for every $t \in T$, since this polynomial vanishes at all points of $t(\R)$, it must also vanish at its complexification.
\end{proof}

\begin{proof}[Proof of Theorem \ref{2}]
The proof of Theorem \ref{GSiegel} adapts almost verbatim to this case. Our assumptions allow us to take the admissible integer $s$ to be equal to $n-d$. Therefore, $\Delta_s(V)=\deg(V)$ and the assertion that $\deg(\mathcal{H}) \in \mathcal{R}_{s,\tau_{l-1}}^{l-1}(V)$, for the algebraic set $\mathcal{H}$ constructed in the proof, becomes the assertion that $\deg(\mathcal{H}) \ge \tau_{l-1} \delta(V)^{d-l+1} \deg(V)$. This allows us to apply the induction hypothesis, with the base case given by Lemma \ref{reduciblecase}. Notice that in the proof we can require $\tau_l$ to be sufficiently small since the result will then obviously hold for all larger values of this parameter.
\end{proof}

\section{Estimating the partial degrees}
\label{S5}

The following lemma is clear upon taking a generic linear combination.

\begin{lema}
\label{simplelinear}
Let $S,V_1,\ldots,V_r$ be subsets of $\mathbb{C}^n$. Let $f_1,\ldots,f_r$ be polynomials such that they all vanish on $S$ but each $f_i$ does not vanish identicaly on $V_i$. Then, there is a linear combination $f=c_1 f_1 + \ldots + c_r f_r$, with real coefficients, such that $f$ vanishes on $S$ but does not vanish identically on any $V_i$.
\end{lema}

We will be needing the next definition in the rest of this article.

\begin{defi}
Let $V \subseteq \C^n$ be an irreducible variety and let $0 \le s \le n-d$ be an integer. We say an irreducible variety $V' \subseteq \C^n$ containing $V$ is an $(n-s)$-minimal variety of $V$ if $\dim(V')=n-s$ and every polynomial of $I(V) \setminus I(V')$ has degree at least $\delta_{s+1}(V)$.
\end{defi}

We can deduce the following estimate from the definition of the partial degree.

\begin{lema}
\label{minimalP}
Let $V \subseteq \C^n$ be an irreducible variety of dimension $d$. Then, there exist polynomials $P_1,\ldots,P_{n-d}$ such that, for every $1 \le i \le n-d$, the maximal dimension of an irreducible component of $Z(P_1,\ldots,P_i)$ containing $V$ is $n-i$ and $Z(P_1,\ldots,P_i)$ contains an $(n-i)$-minimal variety of $V$. Furthermore, they can be chosen so that
$$(\deg(P_1),\ldots,\deg(P_{n-d})) = (\delta_1(V), \ldots, \delta_{n-d}(V)) \le (\deg(P_1^{\ast}),\ldots,\deg(P_{n-d}^{\ast}))$$ 
under lexicographical order, for any other set of polynomials $P_1^{\ast},\ldots,P_{n-d}^{\ast}$ satisfying the same conclusions.
\end{lema}

\begin{proof}
We construct the polynomials $P_i$ recursively on $i$, taking $P_1$ to be any polynomial of degree $\delta_1(V)$ vanishing on $V$. Let $V_1,\ldots,V_r$ be the irreducible components of dimension $n-i+1$ containing $V$ that lie inside of $Z(P_1,\ldots,P_{i-1})$. By definition of $\delta_i(V)$, for each $V_j$ we can find a polynomial $f_j$ of degree at most $\delta_i(V)$ that vanishes on $V$ without vanishing on $V_j$. It follows from Lemma \ref{simplelinear} that we can find a linear combination $f$ with real coefficients of these polynomials $f_j$ such that $f$ vanishes on $V$ and cuts each $V_i$ properly. We take $P_i = f$ and assume for contradiction that $\deg(P_i) < \delta_i(V)$. Let $s \le i$ be the smallest integer with $\delta_s(V)=\delta_i(V)$ and let $W_1,\ldots,W_t$ be the irreducible components of $Z(P_1,\ldots,P_{s-1})$ of dimension $n-s+1$ containing $V$. Each $W_m$ then contains one of the previously defined $V_j$ and so we see in particular that $P_i$ cuts each $W_m$ properly. But this means that the maximal dimension of an irreducible component of $Z(P_1,\ldots,P_{s-1},P_i)$ containing $V$ is $n-s$ while $\deg(P_1), \ldots, \deg(P_{s-1}), \deg(P_i) < \delta_{s}(V)$, thus contradicting the definition of the latter quantity. This argument shows that $(\deg(P_1),\ldots,\deg(P_i)) \le (\deg(P_1^{\ast}),\ldots,\deg(P_i^{\ast}))$, under lexicographical order, for any tuple of polynomials $P_i^{\ast}$ as in the statement of the lemma. Notice also that we have shown there is some $V_j$ with every polynomial of $I(V) \setminus I(V_j)$ having degree at least $\delta_i(V)$. Therefore $V_j$ is an $(n-i+1)$-minimal variety of $V$. Iterating this procedure until $i=n-d$ it only remains to show that $Z(P_1,\ldots,P_{n-d})$ contains a $d$-minimal variety, but this is of course $V$ itself. The result follows.
\end{proof}

From Lemma \ref{Bezout} and Lemma \ref{minimalP} we clearly get the following corollary.

\begin{coro}
\label{Bezu}
Every irreducible variety $V$ of dimension $d$ satisfies
$$ \deg(V) \le \prod_{i=1}^{n-d} \delta_i(V).$$
\end{coro}

It turns out that both sides of Corollary \ref{Bezu} are equal up to an $O_n(1)$-constant. In fact, we have the following general result.

\begin{teo}
\label{clases}
Let $V \subseteq \C^n$ be an irreducible variety of dimension $d$ and let $P_1,\ldots,P_{n-d}$ be the corresponding polynomials given by Lemma \ref{minimalP}. Then, for every $1 \le m \le n-d$, every irreducible component $W$ of $Z(P_1,\ldots,P_m)$ having dimension $n-m$ and containing $V$ has degree $\sim_n \prod_{i=1}^m \delta_i(V)$ and satisfies $\delta_i(W) \sim_n \delta_i(V)$ for every $1 \le i \le m$.
\end{teo}

\begin{proof}
We fix $V$ and induct on $m$, the result being obvious for $m=1$ from the definition of $\delta_1(V)$. Let $W$ be as in the statement and write $Z_1,\ldots,Z_r$ for the irreducible components of $Z(P_1,\ldots,P_{m-1})$ of dimension $n-m+1$ that contain $V$, so in particular we know by induction that the result holds for them. 

Since $W$ is an irreducible component of $Z(P_1,\ldots,P_m)$, we know that
$$ \deg(W) \le \prod_{i=1}^{m} \deg(P_i) = \prod_{i=1}^{m} \delta_i(V).$$
Let now $\varepsilon \gtrsim_n 1$ be a sufficiently small constant with respect to $n$ and assume for contradiction that 
\begin{equation}
\label{eas}
\deg(W) \le \varepsilon \prod_{i=1}^{m} \delta_i(V).
\end{equation}
We know by Theorem \ref{GSiegel} (and Lemma \ref{coverN}) that for every $Z_j$ we can find a polynomial $Q_j$ of degree at most 
$$ \lesssim_n \left( \frac{\deg(W)}{\Delta_{s_j}(Z_j)} \right)^{\frac{1}{m-s_j}},$$
vanishing on $W$ without vanishing on $Z_j$, for an appropriate $0 \le s_j < m$. Now notice that by induction we know that $\Delta_{s_j}(Z_j) \sim_n \prod_{i=1}^{s_j} \delta_i(V)$ and therefore, by (\ref{eas}) and Lemma \ref{dorder}, we have that 
$$ \deg(Q_j) \lesssim_n \left( \frac{\varepsilon \prod_{i=1}^{m} \delta_i(V)}{\prod_{i=1}^{s_j} \delta_i(V)} \right)^{\frac{1}{m-s_j}} \lesssim_n \varepsilon^{\frac{1}{m-s_j}} \delta_{m}(V).$$
This quantity is strictly less than $\delta_{m}(V)$, provided $\varepsilon$ was chosen sufficiently small. We know from Lemma \ref{simplelinear} that we can find some real linear combination $Q$ of the $Q_j$ such that $Q$ vanishes on $W$ without vanishing on any $Z_j$. But this means that $Z(P_1,\ldots,P_{m-1},Q)$ contains $W$, and in particular $V$, while each of its irreducible components containing $V$ has dimension at most $n-m$. Since $\deg(Q) < \delta_{m}(V)=\deg(P_{m})$, this gives us a contradiction by the definition of the polynomials $P_1,\ldots,P_{n-d}$. We have thus shown that
\begin{equation}
\label{ds}
\deg(W) \sim_n  \prod_{i=1}^{m} \delta_i(V).
\end{equation}
Finally, for every $1 \le i \le m$, we have that $V \subseteq W \subseteq Z(P_1,\ldots,P_i)$, from where we know that $\delta_i(W) \le \delta_i(V)$. Combining this with (\ref{ds}) and Corollary \ref{Bezu} applied to $W$, we conclude that it must be $\delta_i(W) \sim_n \delta_i(V)$ for every $1 \le i \le m$, as desired.
\end{proof}

Although we shall not need it, the following is an easy consequence of Theorem \ref{clases} for real varieties.

\begin{coro}
Let $V \subseteq \C^n$ be an irreducible variety of dimension $d$. If $V$ is the complexification of $V(\R)$, then there exist real polynomials $f_1,\ldots,f_{n-d}$ with $\prod_{i=1}^{n-d} \deg(f_i) \lesssim_n \deg(V)$ such that $V$ is an irreducible component of $Z(f_1,\ldots,f_{n-d})$.
\end{coro}

Theorem \ref{clases} also allows us to deduce the following lower bound for the degree of an $(n-k)$-dimensional variety containing $V$.

\begin{coro}
\label{SB}
Consider an irreducible variety $V \subseteq \C^n$ of dimension $d$ and let $W$ be an irreducible variety of dimension $n-k$ containing $V$. Then $\deg(W) \gtrsim_n \delta_1(V) \cdots \delta_k(V)$.
\end{coro}

\begin{proof}
By definition of the quantities $\delta_s(V)$ and Lemma \ref{simplelinear} we can find polynomials $f_{k+1},\ldots,f_{n-d}$ such that for every $k < s \le n-d$ it is $\deg(f_s) \le \delta_s(V)$ and the maximal dimension of an irreducible component of $W \cap Z(f_{k+1},\ldots,f_s)$ containing $V$ is equal to $n-s$. In particular, we have that $W \cap Z(f_{k+1},\ldots,f_{n-d})$ contains $V$ as an irreducible component and therefore must have degree at least $\deg(V) \sim_n \prod_{i=1}^{n-d} \delta_i(V)$ by Theorem \ref{clases}. On the other hand, by Lemma \ref{Bezout} the degree of this algebraic set is bounded by
$$ \deg(W) \prod_{i=k+1}^{n-d} \delta_i(V).$$
The result follows.
\end{proof}

\section{Envelopes and full covers}
\label{S6}

\subsection{Envelopes} We know by Lemma \ref{minimalP} and Theorem \ref{clases} that, given an irreducible variety $V$ of dimension $d$, we can find polynomials $f_1,\ldots,f_{n-d}$ satisfying a good upper bound on their degrees and with $V$ an irreducible component of $Z(f_1,\ldots,f_{n-d})$. On the other hand, the methods of Barone and Basu \cite{BB0,BB} are well suited to study the number of connected components of $Z(f_1,\ldots,f_{n-d})(\R)$ as long as the points of these components have local real dimension at most $n-i$ in $Z(f_1,\ldots,f_i)$ for every $i$. In order to be able to obtain a result in the general case, we therefore need to be able to keep track of those irreducible components of $Z(f_1,\ldots,f_{i})$ that have dimension strictly larger than $n-i$. Our aim is to show that as long as we allow some further flexibility in the upper bound for the degrees of the $f_i$, we can obtain a good amount of control on these higher dimensional components. Of course, the flexibility we wish to allow has to be limited if we expect to obtain optimal bounds. This motivates the following definition.

\begin{defi}[Admissible tuples]
Let $V \subseteq \C^n$ be an irreducible variety of dimension $d$ and $1 \le K_1 \le \ldots \le K_{n-d}$ real numbers. We say an ordered tuple of polynomials $Q=\left\{ Q_1, \ldots, Q_{n-d} \right\}$ is $(K_1,\ldots,K_{n-d})$-admissible for $V$ if, for every $1 \le i \le n-d$, we have that $\deg(Q_i) \le K_i \delta_i(V)$, the maximal dimension of an irreducible component of $Z(Q_1,\ldots,Q_i)$ containing $V$ is equal to $n-i$ and $Z(Q_1,\ldots,Q_i)$ contains an $(n-i)$-minimal variety of $V$. If $Q$ is $(K_1,\ldots,K_{n-d})$-admissible for $V$ and $K \ge K_{n-d}$ then we may simply say $Q$ is $K$-admissible for $V$.
\end{defi}

We have the following observation.

  \begin{lema}
  \label{lower}
  Let $V \subseteq \C^n$ be an irreducible variety of dimension $d$ and $Q= \left\{ Q_1,\ldots,Q_{n-d} \right\}$ a $(K_1,\ldots,K_{n-d})$-admissible tuple of polynomials for $V$. Let $W$ be an irreducible component of $Z(Q_1,\ldots,Q_i)$ of dimension $n-i$, for some $1 \le i \le n-d$. Then, if $\deg(W) \ge c \delta_1(V) \cdots \delta_i(V)$, it must be $\delta_j(W) \sim_{K_{i},c,n} \delta_j(V)$ for every $ 1 \le j \le i$.
  \end{lema}
  
  \begin{proof}
  This follows from the fact that $\deg(W) \sim_n \prod_{j=1}^{i} \delta_j(W)$ by Theorem \ref{clases} and that by construction, for every $1 \le j \le i$, the maximal dimension of an irreducible component of $Z(Q_1,\ldots,Q_{j})$ containing $W$ must be equal to $n-j$, implying that $\delta_j(W) \lesssim_{K_{j},n} \delta_j(V)$.
  \end{proof}
  
  In order to keep track of the higher-dimensional components arising from an admissible tuple, we introduce the following definitions.

\begin{defi}[Envelopes]
Let $V \subseteq \C^n$ be an irreducible variety of dimension $d$ and $Q= \left\{ Q_1,\ldots,Q_{n-d} \right\}$ a $K$-admissible tuple of polynomials for $V$, for some $K \ge 1$.  For each $1 \le j \le n-d$, we define the $j$-th envelope of $V$ over $Q$ to be the union of all irreducible components of $Z(Q_1,\ldots,Q_j)$ of dimension strictly greater than $n-j$ and write $\mathcal{E}_V^{(j)}(Q)$ for this algebraic set. We also write 
$$\mathcal{E}_{V}(Q) = \bigcup_{j=1}^{n-d} \mathcal{E}_{V}^{(j)}(Q),$$
and call this algebraic set the envelope of $V$ over $Q$. Finally, for every $1 \le j \le n-d$, we write $\mathcal{S}_V^{(j)}(Q)$ for the algebraic set given by the union of the irreducible components of $Z(Q_1,\ldots,Q_{j})$ of dimension $n-j$.
\end{defi} 

We begin with some immediate consequences of these definitions.
  
  \begin{lema}
  \label{location}
  If $W$ is an irreducible component of $\mathcal{E}_V(Q)$ of dimension $i$, then it is also an irreducible component of $\mathcal{S}_V^{(n-i)}(Q)$.
  \end{lema}
  
  \begin{proof}
  By construction of $\mathcal{E}_V(Q)$ we know that $W$ must be an irreducible component of $Z(Q_1,\ldots,Q_{j})$ for some $j>n-i$. In particular, $W \subseteq Z(Q_1,\ldots,Q_{n-i})$. But any irreducible component of $Z(Q_1,\ldots,Q_{n-i})$ containing $W$ properly has dimension strictly greater than $i$ and therefore belongs to $\mathcal{E}_V^{(n-i)}(Q) \subseteq \mathcal{E}_V(Q)$, contradicting the fact that $W$ is an irreducible component of this last algebraic set. The result follows.
  \end{proof}
  
  \begin{lema}
  \label{SW}
  Let $W_1,\ldots,W_r$ be the irreducible components of $\mathcal{S}_V^{(i)}(Q)$ for some $1 \le i \le n-d$. Then 
  $$ \sum_{j=1}^r \deg(W_j) \lesssim_{K,n} \delta_1(V) \cdots \delta_i(V).$$
  \end{lema}
  
  \begin{proof}
  By Lemma \ref{Bezout}, this follows immediately from the fact that each irreducible component of $\mathcal{S}_V^{(i)}(Q)$ is also an irreducible component of $Z(Q_1,\ldots,Q_i)$ and $\deg(Q_j) \le K \delta_j(V)$ for every $1 \le j \le i$.
  \end{proof}
  
  \begin{coro}
  \label{LC}
  Let $W_1,\ldots,W_r$ be the irreducible components of $\mathcal{E}_V(Q)$ of dimension $n-k$. Then 
  $$ \sum_{i=1}^r \deg(W_i) \lesssim_{K,n} \delta_1(V) \cdots \delta_k(V).$$
  \end{coro}
  
  \begin{proof}
  This follows from Lemma \ref{location} and Lemma \ref{SW}.
  \end{proof}
  
  The following result shows how we are able to gain control on the irreducible components of an envelope of $V$ at the cost of increasing the allowed upper bound on the degree of the associated admissible tuple.
  
   \begin{prop}
   \label{RC}
  Let $V \subseteq \C^n$ be an irreducible variety of dimension $d$ and let $1=\varepsilon_0 \ge \varepsilon_1 \ge \ldots \ge \varepsilon_{n-d-1} > 0$ be given. Then, there exist constants $C_1,\ldots,C_{n-d}$ with $C_i=O_{n,\varepsilon_{i-1}}(1)$ and a $(C_1,\ldots,C_{n-d})$-admissible tuple of polynomials $Q=\left\{ Q_1,\ldots,Q_{n-d} \right\}$ for $V$, such that, for every $1 \le i < n-d$, the union of all irreducible components of $\mathcal{E}_V(Q)$ of dimension $n-i$ has degree less than $\varepsilon_i \delta_1(V) \cdots \delta_i(V)$.
  \end{prop}
  
  \begin{proof}
  Let $P_1,\ldots,P_{n-d}$ be polynomials of the form given by Lemma \ref{minimalP}. We know that for every $1 \le i \le n-d$ there is an irreducible component of $Z(P_1,\ldots,P_i)$ of dimension $n-i$ containing $V$ that is an $(n-i)$-minimal variety of $V$. Using Lemma \ref{location} we see that it suffices to recursively construct the elements of a tuple of polynomials $Q=\left\{ Q_1,\ldots,Q_{n-d} \right\}$ such that, for every $1 \le i \le n-d$, it is $\deg(Q_i) \le C_i \delta_i(V)$, the maximal dimension of an irreducible component of $Z(Q_1,\ldots,Q_i)$ containing $V$ is equal to $n-i$, $Z(Q_1,\ldots,Q_i)$ contains every $(n-i)$-dimensional component of $Z(P_1,\ldots,P_i)$ containing $V$ and $Q_i$ does not vanish identically on any union of components of $\mathcal{S}_V^{(i-1)}(Q)$ having degree at least $\varepsilon_{i-1} \delta_1(V) \cdots \delta_{i-1}(V)$. Notice that this last requirement is well-defined for recursive purposes since $\mathcal{S}_V^{(i-1)}(Q)$ depends only on the first $i-1$ elements of the tuple $Q$.
  
  We set $Q_1=P_1$, which clearly satisfies the required conditions. Recursively, suppose we have constructed elements $Q_1,\ldots,Q_{i-1}$ with the properties specified in the previous paragraph. By Lemma \ref{simplelinear} we know there is some real linear combination $f$ of the polynomials $P_1,\ldots,P_i$ such that the maximal dimension of an irreducible component of $Z(Q_1,\ldots,Q_{i-1},f)$ containing $V$ is equal to $n-i$ and such that $Z(Q_1,\ldots,Q_{i-1},f)$ contains every $(n-i)$-dimensional irreducible component of $Z(P_1,\ldots,P_i)$ containing $V$.
   
 Let $A$ be the algebraic set of all irreducible components of $Z(Q_1,\ldots,Q_{i-1},f)$ of dimension $n-i$ containing $V$. Notice that by Lemma \ref{Bezout} we know that $\deg(A) \le (\prod_{j=1}^{i-1} C_j)\delta_1(V) \cdots \delta_i(V)$ while Corollary \ref{SB} implies that $\deg(A) \gtrsim_n \delta_1(V) \cdots \delta_i(V)$. Let $W_1,\ldots,W_r$ be the set of irreducible components of $\mathcal{S}_V^{(i-1)}(Q)$. Consider an arbitrary subset $J \subseteq \left\{ 1 ,\ldots, r \right\}$ and let
 $$ W_J = \bigcup_{j \in J} W_j.$$
 Assume $\deg(W_J) \ge \varepsilon_{i-1} \delta_1(V) \cdots \delta_{i-1}(V)$. Notice that since each $W_j$ is an irreducible component of $Z(Q_1,\ldots,Q_{i-1})$ we have in particular that $\delta(W_J) \lesssim_{n,\varepsilon_{i-2}} \delta_{i-1}(V)$. Furthermore, since 
 $$\deg(W_J) \le \prod_{j=1}^{i-1} \deg(Q_j) \lesssim_{n,\varepsilon_{i-2}} \prod_{j=1}^{i-1} \delta_j(V),$$
 we conclude that 
 $$\deg(A) \gtrsim_{n,\varepsilon_{i-2}} \deg(W_J) \delta_i(V) \gtrsim_{n,\varepsilon_{i-2}} \deg(W_J) \delta(W_J).$$ 
  We can therefore apply Theorem \ref{2} to find a polynomial $f_J$ of degree 
   $$ \lesssim_{n, \varepsilon_{i-2}} \frac{\deg(A)}{\deg(W_J)} \lesssim_{n,\varepsilon_{i-1}} \delta_i(V),$$
   vanishing on $A$ without vanishing identically on $W_J$. By Lemma \ref{simplelinear} we can then find a real linear combination $Q_i$ of the polynomials $f_J$ over all $J$ with $\deg(W_J) \ge \varepsilon_{i-1} \delta_1(V) \cdots \delta_{i-1}(V)$ such that $Q_i$ vanishes on $A$ without vanishing identically on any such $W_J$. The result follows.
  \end{proof}
  
 We now deduce from Proposition \ref{RC} that we can find an envelope of $V$ whose irreducible components lie on the zero set of some appropriate polynomials of small degree.
  
  \begin{lema}
  \label{env2}
Let $V \subseteq \C^n$ be an irreducible variety of dimension $d$. Then there exists an $O_{n}(1)$-admissible tuple $Q=\left\{Q_1,\ldots,Q_{n-d} \right\}$ of polynomials for $V$ and polynomials $F_1,\ldots,F_{n-d-1}$ with $\deg(F_k) < \delta_k(V)$ for every $1 \le k \le n-d-1$, such that $F_k$ vanishes on all the irreducible components of $\mathcal{E}_V(Q)$ of dimension $n-k$ without vanishing identically on $V$.
\end{lema}

\begin{proof}
Let $\varepsilon_1 > \ldots > \varepsilon_{n-d-1}>0$ be real numbers such that each $\varepsilon_i$ is sufficiently small with respect to $n$ and $\varepsilon_{i-1}$. Let $Q=\left\{ Q_1, \ldots, Q_{n-d} \right\}$ be the tuple given by Proposition \ref{RC} with respect to these parameters. Let us write $A_k$ for the union of the irreducible components of $\mathcal{E}_V(Q)$ of dimension $n-k$. By Proposition \ref{RC} we know that $\deg(A_k) \le \varepsilon_k \delta_1(V) \cdots \delta_k(V)$. Let now $V_{n-k+1}$ be an $(n-k+1)$-minimal variety of $V$ contained inside of $Z(Q_1,\ldots,Q_{k-1})$, which we know to exist by definition of an admissible tuple. By Theorem \ref{GSiegel} (and Lemma \ref{coverN}) we can find a polynomial $F_k$ of degree
$$ \lesssim_n \left( \frac{\deg(A_k)}{\Delta_s(V_{n-k+1})} \right)^{\frac{1}{k-s}},$$
vanishing on $A_k$ without vanishing identically on $V_{n-k+1}$, for some $0 \le s< k$. Since $\Delta_s(V_{n-k+1}) \sim_{n,\varepsilon_{k-1}} \delta_1(V) \cdots \delta_s(V)$ by Corollary \ref{SB} and Lemma \ref{lower}, it follows that the above quantity is 
$$\lesssim_{n,\varepsilon_{k-1}} \varepsilon_k^{1/(k-s)} \delta_{k}(V).$$
This is strictly less than $\delta_{k}(V)$ provided $\varepsilon_k>0$ was chosen sufficiently small with respect to $n$ and $\varepsilon_{k-1}$. Since $V_{n-k+1}$ is an $(n-k+1)$-minimal variety, this means that $F_k$ cannot vanish identically on $V$ and the result follows.
\end{proof}
  
  \subsection{Full covers}
  
  We have seen how we can obtain some control on the irreducible components of an envelope of $V$. In order to take advantage of this, we will need to find an admissible tuple of polynomials for each of these irreducible components and, once again, control the higher-dimensional components that they produce. In order to handle this recursive procedure, we introduce the following definition. 
  
  \begin{defi}[Full cover]
  Let $V \subseteq \C^n$ be an irreducible variety and $K>0$. If $\dim (V) = n-1$, we say an algebraic set $\mathcal{F}(V)$ is a $K$-full cover of $V$ if $\mathcal{F}(V)=V$. Recursively, let $\dim(V)=d<n-1$. We say an algebraic set $\mathcal{F}(V)$ is a $K$-full cover of $V$ if there exists a $K$-admissible tuple of polynomials $Q$ for $V$ such that 
  $$ \mathcal{F}(V) = \mathcal{S}_V^{(n-d)}(Q) \cup \bigcup_{W_i \subseteq \mathcal{E}_V(Q)} \mathcal{F}(W_i),$$
  where the last union runs through all the irreducible components $W_i$ of $\mathcal{E}_V(Q)$ and each $\mathcal{F}(W_i)$ is a $K$-full cover of $W_i$.
  \end{defi}
  
  The next lemma shows an appropriate bound for the degree of a full cover.
  
  \begin{lema}
  \label{FD}
  Let $\mathcal{F}(V)$ be a $K$-full cover of $V$ and let $W_1,\ldots,W_r$ be the irreducible components of $\mathcal{F}(V)$. Then
  $$ \sum_{i=1}^r \deg(W_i) \lesssim_{K,n} \deg(V).$$
  \end{lema}
  
  \begin{proof}
  We proceed by induction on the codimension, the claim being trivial for hypersurfaces. So let $\dim(V) \le n-2$ and assume the result holds for all varieties of strictly larger dimension. The irreducible components of $\mathcal{F}(V)$ that are also irreducible components of $\mathcal{S}_V^{(n-d)}(Q)$ satisfy the bound by Lemma \ref{SW}. Every other irreducible component is an irreducible component of a $K$-full cover $\mathcal{F}(W)$ of some irreducible component $W$ of $\mathcal{E}_V(Q)$, so in particular $\dim(W)>\dim(V)$. For a fixed choice of $W$, we know by induction that the degrees of the irreducible components of $\mathcal{F}(W)$ sum up to at most $\lesssim_{K,n} \deg(W)$. The result then follows, by Corollary \ref{LC}, upon summing over all irreducible components $W$ of $\mathcal{E}_V(Q)$.
  \end{proof}
  
  We now show how the bounds we have attained on the envelopes allow us to obtain an appropriate full cover.
  
  \begin{prop}
  \label{EXI}
  Let $V \subseteq \C^n$ be an irreducible variety of dimension $d$ and let $\varepsilon > 0$ be given. Then $V$ admits an $O_{\varepsilon,n}(1)$-full cover $\mathcal{F}(V)$ such that, for every $1 \le k < n-d$, every irreducible component of $\mathcal{F}(V)$ of dimension $n-k$ has degree at most $\varepsilon \delta_1(V) \cdots \delta_{k}(V)$.
  \end{prop}
 
 \begin{proof}
  We proceed by induction on the codimension, the result being clear when $d=n-1$. Let $\varepsilon=\varepsilon_0 > \varepsilon_1 > \cdots > \varepsilon_{n-d-1}> 0$ be chosen so that $\varepsilon_i$ is sufficiently small with respect to $n$ and $\varepsilon_{i-1}$. Let $Q$ be a $(C_1,\ldots,C_{n-d})$-admissible tuple of polynomials for $V$ of the form given by Proposition \ref{RC} with respect to the parameters $\varepsilon_i$ we have just defined. Since every irreducible component of $\mathcal{S}_V^{(n-d)}(Q)$ has dimension $d$, it suffices to show that for every irreducible component $W$ of $\mathcal{E}_V(Q)$ we can find a full cover $\mathcal{F}(W)$ with all its irreducible components of dimension $n-k \ge \dim(W)$ having degree at most $\varepsilon \delta_1(V) \cdots \delta_k(V)$. Fix such a choice of $W$. By Lemma \ref{location} we know that $W$ is an irreducible component of $Z(Q_1,\ldots,Q_{n-\dim(W)})$ and from this it follows that $\delta_i(W) \lesssim_{\varepsilon_{i-1},n} \delta_i(V)$ for all $1 \le i \le n-\dim(W)$. Let now $\varepsilon_W \gtrsim_{n,\varepsilon_{n-\dim(W)-1}}1$ be a sufficiently small constant with respect to $\varepsilon_{n-\dim(W)-1}$ and $n$. By induction, we know that we can find an $O_{\varepsilon_W,n}(1)$-full cover $\mathcal{F}(W)$ of $W$ such that, if $n-k>\dim(W)$, then all irreducible components of $\mathcal{F}(W)$ of dimension $n-k$ have degree at most 
  \begin{equation}
  \begin{aligned}
   \varepsilon_W \delta_1(W) \cdots \delta_k(W) &\lesssim_{n,\varepsilon_{k-1}} \varepsilon_W \delta_1(V) \cdots \delta_k(V) \\
    &\le \varepsilon \delta_1(V) \cdots \delta_k(V),
   \end{aligned}
   \end{equation}
  upon choosing $\varepsilon_W$ sufficiently small with respect to $n$ and $\varepsilon_{n-\dim(W)-1}$. Finally, the irreducible components of $\mathcal{F}(W)$ having dimension $\dim(W)$ have degree at most $\lesssim_{n,\varepsilon_W} \deg(W)$ by Lemma \ref{SW} and, since $W$ is an irreducible component of $\mathcal{E}_V(Q)$ and we have chosen the tuple $Q$ to satisfy Proposition \ref{RC}, this is at most 
  $$ \lesssim_{n,\varepsilon_W} \varepsilon_{n-\dim(W)} \delta_1(V) \cdots \delta_{n-\dim(W)}(V) \le \varepsilon \delta_1(V) \cdots \delta_{n-\dim(W)}(V),$$
   upon choosing $\varepsilon_{n-\dim(W)}$ sufficiently small with respect to $\varepsilon_W \gtrsim_{n,\varepsilon_{n-\dim(W)-1}}1$ and therefore with respect to $ \varepsilon_{n-\dim(W)-1}$ and $n$. The result follows.
  \end{proof}
  
  We get the following important consequence of the previous proposition.
  
  \begin{coro}
  \label{mt}
  Every irreducible variety $V \subseteq \C^n$ admits an $O_n(1)$-full cover $\mathcal{F}(V)$ such that $V$ is an irreducible component of $\mathcal{F}(V)$.
  \end{coro}
  
  \begin{proof}
Let $\varepsilon> 0$ be sufficiently small with respect to $n$. Let $\mathcal{F}(V)$ be an $O_{\varepsilon,n}(1)$-full cover of $V$ of the form provided by Proposition \ref{EXI} with respect to $\varepsilon$ and suppose $\mathcal{F}(V)$ has some irreducible component $W$ of dimension $n-k>\dim(V)$ containing $V$. Then we have that $\deg(W) \le \varepsilon \delta_1(V) \cdots \delta_k(V)$, but this contradicts Corollary \ref{SB} if $\varepsilon$ was chosen sufficiently small with respect to $n$.
\end{proof}

\section{Bounding the number of connected components}
\label{S7}

\subsection{A result of Barone and Basu}

In the proofs of Theorem \ref{4} and Theorem \ref{5} given below, a crucial role will be played by a result of Barone and Basu \cite{BB}. We begin this section by introducing some notation that will help us state this result. 

Given an algebraic set $X \subseteq \C^n$ and a point $x \in X(\R)$, we write $\dim_x^{\R}X$ for the local real dimension of $X(\R)$ at $x$ \cite[\S 2.8]{BCR}. If $Q=\left\{ Q_1,\ldots,Q_m \right\}$, $Q_1,\ldots,Q_m \in \C[x_1,\ldots,x_n]$, is a tuple of polynomials with $m \le n$, then for every $1 \le j \le m$ we shall write
$$ Z_j(Q) = Z(Q_1,\ldots,Q_j),$$
and if $x \in Z_j(Q)(\R)$, we let
$$\dim_{Q,(j)}^{\R}(x) = \left( \dim_x^{\R} Z_1(Q), \ldots, \dim_x^{\R} Z_j(Q)\right).$$
Given two tuples of non-negative integers $\tau = (\tau_1,\ldots,\tau_j)$ and $\sigma = (\sigma_1, \ldots, \sigma_j)$, we will write $\tau \le \sigma$ if $\tau_i \le \sigma_i$ for every $ 1 \le i \le j$. Given a $j$-tuple $\tau$ of non-negative integers with $1 \le j \le m$ and a set of polynomials $Q=\left\{ Q_1,\ldots,Q_m \right\}$ we can consider the set
$$ Z_{\tau}(Q) = \overline{\left\{ x \in Z_{j}(Q)(\R) : \dim_{Q,(j)}^{\R}(x) = \tau \right\}},$$
where the closure is taken with respect to the Euclidean topology of $\R^n$. Clearly $Z_{\tau}(Q) \subseteq Z_j(Q)(\R)$ for every $j$-tuple $\tau$.

We have the following result of Barone and Basu \cite{BB}.

\begin{prop}
\label{BBlema0}
Let $Q= \left\{ Q_1,\ldots,Q_m \right\}$, $Q_1,\ldots,Q_m \in \C[x_1,\ldots,x_n]$, $ 1 \le m \le n$, be a set of polynomials and $d_1 \le \ldots \le d_m$ integers with $\deg(Q_i) \le d_i$ for every $i$. Let $1 \le j \le m$ and let $\tau=(\tau_1,\ldots,\tau_j)$, $\tau_1 \ge \ldots \ge \tau_j$, be a $j$-tuple of non-negative integers satisfying $\tau_i \le n-i$ for every $1 \le i \le j$. Then, the number of connected components of $Z_j(Q)(\R)$ intersecting $Z_{\tau}(Q)$ is at most $\lesssim_n (\prod_{i=1}^j d_i ) d_j^{n-j}$.
\end{prop}

\begin{proof}
Let us write $Q_i=p_i+{\bf i} q_i$ with $p_i,q_i$ real polynomials and {\bf i} the imaginary unit and consider the real polynomial $|Q_i|^2 = Q_i \overline{Q_i} = p_i^2 + q_i^2$. Notice that $Z(Q_i)(\R)=Z(|Q_i|^2)(\R)$ and in particular, writing $|Q|^2=\left\{ |Q_1|^2,\ldots,|Q_m|^2 \right\}$, it is $Z_j(Q)(\R) = Z_j(|Q|^2)(\R)$ and $Z_{\tau}(Q) = Z_{\tau}(|Q|^2)$ for every $j$-tuple $\tau$. It therefore suffices to establish the result with the tuple of real polynomials $|Q|^2$ instead of $Q$. Let now $W_{\tau}'$ be as in \cite[Corollary 3.19]{BB}, so that we have $Z_{\tau}(|Q|^2) \subseteq W_{\tau}' \subseteq Z_j(|Q|^2)(\R)$ in our setting. In particular, to each connected component of $Z_j(|Q|^2)(\R)$ intersecting $Z_{\tau}(|Q|^2)$ we can assign a connected component of $W_{\tau}'$ intersected by it. Such an assignment will be injective, since $W_{\tau}' \subseteq Z_j(|Q|^2)(\R)$. Therefore, it suffices to obtain the corresponding estimate for $W_{\tau}'$ and this follows from Corollary 3.25, Lemma 3.27 and Lemma 3.28 of \cite{BB}. Notice that strictly speaking, the cited results from \cite{BB} assume boundedness over a real closed field, which can be accomplished as in Section 3.3 of that article. 
\end{proof}

We will use the following observation.

\begin{lema}
\label{cboundsr}
Let $V \subseteq \C^n$ be an irreducible variety of dimension $d$ and $Q$ a $K$-admissible tuple of polynomials for $V$. Let $x \in Z_j(Q)(\R) \setminus \mathcal{E}_V^{(j)}(Q)$. Then
$$\dim^{\R}_{Q,(j)}(x) \le (n-1,n-2,\ldots,n-j).$$
\end{lema}

\begin{proof}
This follows from the fact that the complex dimension of a variety bounds the real dimension and the definition of $\mathcal{E}_V^{(j)}(Q)$.
\end{proof}

We can deduce the following result from Proposition \ref{BBlema0}.

\begin{coro}
\label{BBlema}
Let $V$ be an irreducible variety of dimension $d$ and $Q$ a $K$-admissible tuple of polynomials for $V$. Let $1 \le j \le n-d$ and let $\tau=(\tau_1,\ldots,\tau_j)$, $\tau_1 \ge \ldots \ge \tau_j$, be a $j$-tuple of non-negative integers satisfying $\tau_i \le n-i$ for every $1 \le i \le j$. Then, the number of connected components of $Z_j(Q)(\R)$ intersecting $Z_{\tau}(Q)$ is at most $\lesssim_{n,K} \delta(V)^d \deg(V)$.
\end{coro}

\begin{proof}
By Proposition \ref{BBlema0} we know that the number of such components is at most 
\begin{equation}
\begin{aligned}
\label{b0W}
&\lesssim_{n,K} \left( \prod_{i=1}^j \delta_i(V) \right) \delta_j(V)^{n-j} \\
 &\lesssim_{n,K} (\delta_1(V) \cdots \delta_{n-d}(V) ) \delta_{n-d}(V)^d\\
 &\lesssim_{n,K}  \deg(V) \delta(V)^d,
 \end{aligned}
 \end{equation}
as desired.
\end{proof}

\subsection{Proof of Theorem \ref{4}}

Theorem \ref{4} follows from Lemma \ref{FD}, Corollary \ref{mt} and the following result.

\begin{teo}
Let $V \subseteq \C^n$ be an irreducible variety of dimension $d$ and let $\mathcal{F}(V)$ be a $K$-full cover of $V$. Then the number $b_0(\mathcal{F}(V)(\R))$ of connected components of $\mathcal{F}(V)(\R)$ satisfies
$$ b_0(\mathcal{F}(V)(\R)) \lesssim_{K,n} \deg(V) \delta(V)^d.$$
\end{teo} 

\begin{proof}
We know the result holds for hypersurfaces by \cite{M}, so we may assume $d<n-1$ and that the result holds for all irreducible varieties of dimension strictly larger than $\dim(V)$. Let $Q=(Q_1,\ldots,Q_{n-d})$ be the $K$-admissible tuple of polynomials for $V$ associated to $\mathcal{F}(V)$. Write 
\begin{equation}
\label{Ydef}
Y = \bigcup_{W_i} \mathcal{F}(W_i),
\end{equation}
with the union going through all irreducible components of $\mathcal{E}_V(Q)$ and where $\mathcal{F}(W_i)$ is the $K$-full cover of $W_i$ associated to $\mathcal{F}(V)$, so 
$$\mathcal{F}(V) = \mathcal{S}_V^{(n-d)}(Q) \cup Y.$$
Write $X=\mathcal{S}_V^{(n-d)}(Q) \setminus Y$. If $x \in X(\R)$, since $\mathcal{E}_V(Q) \subseteq Y$, we know by Lemma \ref{cboundsr} that 
$$ \dim_{Q,(n-d)}^{\R} (x) \le (n-1, n-2, \ldots, d).$$ 
In particular, if $C$ is a connected component of $\mathcal{F}(V)(\R)$ that intersects $X$, then it also intersects $Z_{\tau}(Q)$ for some $(n-d)$-tuple $\tau \le (n-1,n-2,\ldots,d)$. Since 
$$Z_{\tau}(Q) \subseteq Z_{n-d}(Q) \subseteq \mathcal{F}(V),$$
we see that the number of connected components of $\mathcal{F}(V)(\R)$ that intersect $X$ is at most the sum over all $\tau \le (n-1,n-2,\ldots,d)$ of the number of connected components of $Z_{n-d}(Q)(\R)$ that intersect $Z_{\tau}(Q)$. By Corollary \ref{BBlema} this quantity is at most $\lesssim_{K,n} \deg(V) \delta(V)^d$. 

It remains to estimate the number of connected components of $\mathcal{F}(V)(\R)$ that do not intersect $X$. These are all contained inside of $Y \subseteq \mathcal{F}(V)$ and so it suffices to estimate the number of connected components of $Y(\R)$. By (\ref{Ydef}) and the induction hypothesis, this quantity is at most
\begin{equation}
\label{tero}
 \lesssim_{K,n} \sum_{W_i} \deg(W_i) \delta(W_i)^{\dim(W_i)},
 \end{equation}
where again the sum runs along all irreducible components of $\mathcal{E}_V(Q)$. Suppose that $W_i$ has dimension $n-k$. Then by Lemma \ref{location} we know that it is an irreducible component of $S_V^{(k)}(Q)$, and therefore of $Z_k(Q)$, and so in particular satisfies $\delta(W_i) \lesssim_{K,n} \delta_k(V)$. Combining this fact with Corollary \ref{LC}, we conclude that the sum (\ref{tero}) restricted to all $W_i$ of dimension $n-k$ is at most
$$ \lesssim_{K,n} \delta_1(V) \cdots \delta_k(V) \delta_k(V)^{n-k} \lesssim_{K,n} \deg(V) \delta(V)^d.$$
Summing this among all $O(1)$ choices of $k$, we obtain the desired result.
\end{proof}

\subsection{Proof of Theorem \ref{5}}

We will establish Theorem \ref{5} by induction on $d$, the result being clear when $d=0$. Assume first that $\deg(P) \le \delta_1(V)$. Then we already know by \cite{OP} that the number of connected components of $\R^n \setminus Z(P)$ itself is bounded by
\begin{equation}
\begin{aligned}
 \lesssim_n \deg(P)^n &\lesssim_n \delta_1(V) \cdots \delta_{n-d}(V) \deg(P)^d \\
 &\lesssim_n \deg(V) \deg(P)^d,
 \end{aligned}
 \end{equation}
so we are done in this case. Let then $\delta_j(V)$ be the largest integer $1 \le j \le n-d$ satisfying $\deg(P) \ge \delta_j(V)$. Let $Q$ be an $O_{n}(1)$-admissible tuple for $V$ of the form provided by Lemma \ref{env2}.

Let us deal first with the points inside of $(V \cap \mathcal{E}_V^{(j)}(Q))(\R)$. Notice that by definition every point of $\mathcal{E}_V^{(j)}(Q)$ lies inside an irreducible component of $\mathcal{E}_V(Q)$ of dimension strictly larger than $n-j$. Since the tuple $Q$ was chosen to satisfy Lemma \ref{env2}, we then know that we can find a polynomial $F=\prod_{i=1}^{j-1} F_i$ of degree 
$$\lesssim_n \delta_{j-1}(V) \lesssim_n \deg(P),$$
vanishing on $\mathcal{E}_V^{(j)}(Q)$ without vanishing identically on $V$. In particular, it suffices to estimate the number of connected components of $\R^n \setminus Z(P)$ intersected by $V \cap Z(F)$. Let $W_1,\ldots,W_r$ be the irreducible components of this intersection, all of which have dimension $d-1$ since $Z(F)$ intersects $V$ properly. We therefore know by induction that the number of connected components of $\R^n \setminus Z(P)$ intersected by $V \cap Z(F)$ is at most
\begin{equation}
\begin{aligned}
\lesssim_n \sum_{i=1}^r \deg(W_i) \deg(P)^{\dim(W_i)} &\lesssim_n \deg(V) \deg(F) \deg(P)^{d-1} \\
&\lesssim_n \deg(V) \deg(P)^d.
\end{aligned}
\end{equation}

It therefore only remains to bound the number of connected components of $\R^n \setminus Z(P)$ intersected by $V(\R) \setminus Z(F)$. In particular, writing $|F|^2$ for the real polynomial $F \overline{F}$ and $f=P|F|^2$, it will suffice to bound the number of connected components of $\R^n \setminus Z(f)$ intersected by $Z_j(Q)$. It will also suffice without loss of generality to bound the number of such components where $f$ is positive. Let us write $Z_1,\ldots,Z_m$ for the irreducible components of $Z_j(Q)$ of dimension $n-j$ that are not contained inside of $Z(f)$. We will bound first the size of the set $\mathcal{C}$ of all connected components $C$ of $\R^n \setminus Z(f)$ that contain an element $x \in Z_j(Q)$ with $f(x)>0$ that can be joined to $Z(f)$ through a path $\pi_C$ in $Z_j(Q)(\R)$. We may assume without loss of generality that only the endpoint of this path lies inside of $Z(f)$. Notice that since $\mathcal{E}_V^{(j)}(Q) \subseteq Z(F)$, every point but at most the endpoint of this path $\pi_C$ lies inside of $(Z_1 \cup \ldots \cup Z_m) \cap Z_{\tau}(Q)$ for some $j$-tuple $\tau \le (n-1,\ldots,n-j)$ that may depend on the specific point. Also, by the existence of the path $\pi_C$, we have that for each such $C$ there exists some $\varepsilon_C >0$ such that the image of $(Z_1 \cup \ldots \cup Z_m) \cap C$ under $f$ contains the interval $(0,\varepsilon_C)$. In particular, since there are finitely many connected components $C$ of $\R^n \setminus Z(f)$ and finitely many components $Z_1,\ldots,Z_m$, we can find some $\varepsilon>0$ such that $f-\varepsilon$ does not vanish identically on any $Z_i$ and $X= (Z_1 \cup \ldots \cup Z_m) \cap Z(f-\varepsilon)$ intersects every element of $\mathcal{C}$. Notice that an element of $X$ belongs to $Z_{\sigma}( \left\{ Q_1,\ldots,Q_j,f-\varepsilon \right\} )$ for some $(j+1)$-tuple $\sigma \le (n-1,\ldots,n-j-1)$ that may again depend on the specific point. In particular, $|\mathcal{C}|$ is bounded by the number of connected components of $Z(Q_1,\ldots,Q_j,f-\varepsilon)$ intersecting some $Z_{\sigma}( \left\{ Q_1,\ldots,Q_j,f-\varepsilon \right\} )$ of the above form. By Proposition \ref{BBlema0} we therefore conclude that
$$ |\mathcal{C}| \lesssim_n \delta_1(V) \cdots \delta_j(V) \deg(f)^{n-j} \lesssim_n \deg(V) \deg(P)^{d},$$
so we are done in this case.

It remains to bound the number of connected components $C$ of $\R^n \setminus Z(f)$ intersected by $Z_j(Q)$ but such that no element of $Z_j(Q) \cap C$ can be joined to $Z(f)$ through a path inside of $Z_j(Q)(\R)$. Clearly, this means that these components of $\R^n \setminus Z(f)$ properly contain a connected component of $Z_j(Q)(\R)$. This component intersects $Z_{\tau}(Q)$ for some $\tau \le (n-1,\ldots,n-j)$ since it contains an element of $Z_j(Q)(\R) \setminus \mathcal{E}_V^{(j)}(Q)$. Therefore, by Corollary \ref{BBlema}, the number of such components is at most 
$$ \lesssim_n \delta_1(V) \cdots \delta_j(V) \delta_j(V)^{n-j} \lesssim_n \deg(V) \delta(V)^d.$$
The result follows.

\section{The incidence geometry of hypersurfaces} 

\subsection{Free configurations}
\label{81}

Let $S$ be a finite set of points in $\R^n$ and $T$ a finite set of varieties in $\R^n$. Given a point $s \in S$ we write $T_s$ for those elements of $T$ containing $s$ and similarly, given $t \in T$, we write $S_t$ for those points of $S$ lying inside of $t$. Notice that
\begin{equation}
\label{DC}
 \mathcal{I}(S,T) = \sum_{s \in S} |T_s| = \sum_{t \in T} |S_t|.
 \end{equation}
We will use the following definition.

\begin{defi}
\label{free}
We say $S$ is $(k,b)$-free with respect to $T$ if, for every choice of $k$ distinct elements $s_1,\ldots,s_k$ from $S$ and $b$ distinct elements $t_1,\ldots,t_b$ from $T$, we have $s_i \notin t_j$ for some $1 \le i \le k$, $1 \le j \le b$.
\end{defi}

Notice that if $S$ is $(k,b)$-free with respect to $T$, then for every choice of subsets $S' \subseteq S$, $T' \subseteq T$, we have that $S'$ is $(k,b)$-free with respect to $T'$. We will need the following well-known lemma (see the Kövári-Sós-Turán bound \cite{KST}, which is a slight refinement of this).

\begin{lema}
\label{trivialbound}
Let $k,b \ge 1$ be integers such that $S$ is $(k,b)$-free with respect to $T$. Then
$$ \mathcal{I}(S,T) \le b^{1/k} |S||T|^{1-1/k}+ (k-1) |T|.$$
\end{lema}

\begin{proof}
If $k=1$ then we clearly have the bound $\mathcal{I}(S,T) \le b|S|$. Inductively, let $S$ be $(k,b)$-free with respect to $T$ and suppose the result has already been established for all smaller values of $k$. Notice that if $s \in S$ then $S \setminus \left\{s\right\}$ is $(k-1,b)$-free with respect to $T_s$. By induction and (\ref{DC}), it follows that
\begin{equation}
\label{gor1}
\begin{aligned}
\sum_{s \in S} \mathcal{I}(S,T_s) &\le \sum_{s \in S} \mathcal{I}(S \setminus \left\{ s \right\},T_s) + \sum_{s \in S} |T_s| \\
&\le b^{\frac{1}{k-1}} \sum_{s \in S} |S| |T_s|^{1-\frac{1}{k-1}} + \sum_{s \in S} (k-1) |T_s| \\
&\le b^{\frac{1}{k-1}} |S|^{1+\frac{1}{k-1}} \mathcal{I}(S,T)^{1-\frac{1}{k-1}} + (k-1) \mathcal{I}(S,T),
\end{aligned}
\end{equation}
where the last bound follows from Hölder's inequality. On the other hand, given $t \in T$ and (not necessarily distinct) elements $s,s' \in S$ incident to $t$, we see that $(s',t)$ is an incidence counted in $\mathcal{I}(S,T_s)$ and therefore 
\begin{equation}
\label{gor2}
\begin{aligned}
\sum_{s \in S} \mathcal{I}(S ,T_s) &\ge \sum_{t \in T} |S_t|^2 \\
&\ge \frac{1}{|T|} \left( \sum_{t \in T} |S_t| \right)^2 \\
&= \frac{\mathcal{I}(S,T)^2}{|T|},
\end{aligned}
\end{equation}
where we have used the Cauchy-Schwarz inequality. The result follows upon comparing (\ref{gor1}) and (\ref{gor2}).
\end{proof}

Notice that the proof of this result does not use any algebraic property of $S$ or $T$ and so is valid in the more general context of abstract bipartite graphs \cite{KST}.

\subsection{Proof of Theorem \ref{6}}

Assume first $d=1$. If $|S| < k$, then clearly $\mathcal{I}(S,T) \le (k-1) \deg(T)$. Otherwise, since $S$ is $(k,b)$-free with respect to $T$, it follows that there are at most $b-1$ elements of $T$ containing all of $V$, thus contributing at most $(b-1)|S|$ incidences. The remaining elements $t \in T$ intersect $V$ in at most $\deg(t) \deg(V)$ points and therefore contribute at most $\deg(T) \deg(V)$ incidences. This gives the result in this case.

Let now $d > 1$ and assume the result holds for every smaller dimension. By the same argument as before either $\mathcal{I}(S,T) \le (k-1)\deg(T)$ or the subset $T_V \subseteq T$ of those elements of $T$ containing all of $V$ satisfies $|T_V| = b'$ for some $0 \le b' \le b-1$. These elements contribute at most $b'|S|$ incidences, so we will restrict attention to the set $T'=T \setminus T_V$. Notice that every $t \in T'$ satisfies $\dim(t \cap V)<d$. For every integer $0 \le s \le n-d$ we will consider the parameters
$$ M_s = \left( \frac{b|S|^k}{k^k \deg(T) (\prod_{i=1}^s \delta_i(V))^k} \right)^{\frac{1}{k(n-s)-1}}.$$
Notice that if $M_{n-d}$ is sufficiently small with respect to $n$ this implies that
$$ |S| \le k b^{-1/k} \deg(T)^{1/k} \deg(V).$$
Plugging this into Lemma \ref{trivialbound} we get the bound
$$ \mathcal{I}(S,T) \le k \deg(T) \deg(V),$$
which is acceptable. We may therefore assume from now on that $M_{n-d} \gtrsim_n 1$. This allows us to apply Theorem \ref{30par} to find a real polynomial $P \notin I(V)$ with $ \deg(P) \lesssim_n M_{n-d}$, such that, writing $s=i_V(M_{n-d})$, each connected component of $\R^n \setminus Z(P)$ contains
$$ \lesssim_n \frac{|S|}{M_{n-d}^{n-s}\Delta_{s}(V)} \lesssim_n \frac{|S|}{M_{n-d}^{n-s} \prod_{i=1}^{s} \delta_i(V)},$$
elements from $S$. Notice now that since $M_{n-d} \lesssim_n \delta_{s+1}(V)$ we have that
\begin{equation}
\begin{aligned}
M_{n-d} &= \left( \frac{b|S|^k}{k^k \deg(T) (\prod_{i=1}^{s} \delta_i(V))^k} \right)^{\frac{1}{k(n-s)-1}+\left( \frac{1}{dk-1}- \frac{1}{k(n-s)-1} \right)} (\prod_{i=s+1}^{n-d} \delta_i(V))^{-\frac{k}{dk-1}} \\
&= M_s \left( \frac{b|S|^k}{k^k \deg(T) (\prod_{i=1}^{s} \delta_i(V))^k} \right)^{\left( \frac{1}{dk-1}- \frac{1}{k(n-s)-1} \right)} (\prod_{i=s+1}^{n-d} \delta_i(V))^{-\frac{k}{dk-1}} \\
&= M_s \left( \frac{b|S|^k}{k^k \deg(T) (\prod_{i=1}^{n-d} \delta_i(V))^k} \right)^{\left( \frac{1}{dk-1}- \frac{1}{k(n-s)-1} \right)} (\prod_{i=s+1}^{n-d} \delta_i(V))^{-\frac{k}{k(n-s)-1}} \\
&\lesssim_n M_s \delta_{s+1}(V)^{1-\frac{dk-1}{k(n-s)-1} - \frac{(n-d-s)k}{k(n-s)-1}} = M_s.
\end{aligned}
\end{equation}
Let $V_{n-s}$ be an $(n-s)$-minimal variety of $V$, so in particular $S \subseteq V \subseteq V_{n-s}$, $\dim(V_{n-s})=n-s$, $V_{n-s}$ is irreducible and $\deg(V_{n-s}) \sim_n \prod_{i=1}^s \delta_i(V)$. Write $\Omega_1,\ldots,\Omega_g$ for the connected components of $\R^n \setminus Z(P)$, $S_i$ for those elements of $S$ lying inside of $\Omega_i$ and $T_i$ for those elements $t \in T'$ such that $t \cap V_{n-s}$ intersects $\Omega_i$. Notice that since the irreducible components of $t \cap V_{n-s}$ have dimension $n-s-1$ and their degrees sum up to at most $\deg(t) \deg(V_{n-s})$, it follows from Theorem \ref{5} that each element $t \in T'$ belongs to $T_i$ for at most $\lesssim_n \deg(t) \deg(V_{n-s}) \deg(P)^{n-s-1}$ values of $1 \le i \le g$.

Using Lemma \ref{trivialbound} we therefore see that the sum of the incidences occurring in each cell is at most
\begin{equation}
\begin{aligned}
 &\le  \sum_{i=1}^g b^{1/k} |S_i| |T_i|^{1-1/k} + (k-1) |T_i| \\
  &\lesssim_n b^{1/k} \frac{|S|}{(M_{n-d}^{n-s} \deg (V_{n-s}))^{1-1/k}} (\deg(T) \deg(V_{n-s}) M_{n-d}^{n-s-1})^{1-1/k} \\
  &+ k \deg(T) \deg(V_{n-s}) M_{n-d}^{n-s-1} \\
  &\lesssim_n b^{1/k} \frac{|S| \deg(T)^{1-1/k}}{M_{n-d}^{1-1/k}} + k \deg(T) (\prod_{i=1}^{s} \delta_i(V)) M_{n-d}^{n-s-1}.
  \end{aligned}
  \end{equation}
  If $M_{n-d} = M_s$ then both terms are equal. Since $1 \lesssim_n M_{n-d} \lesssim_n M_s$ it follows that the left-hand side dominates up to at most a $O_n(1)$-constant. Expanding $M_{n-d}$ on the left-hand side then shows that the above expression is bounded by
  \begin{equation}
  \label{500}
   \lesssim_n \tau_d(b,k) |S|^{\alpha_k(d)}\deg(T)^{\beta_k(d)} \deg(V)^{1-\alpha_k(d)}.
   \end{equation}
  
 It remains to deal with the incidences coming from $Z(P) \cap V$. Let 
 $$Z(P) \cap V=W_1 \cup \ldots \cup W_r,$$
 be the decomposition into irreducible components, so in particular $\dim(W_i)=d-1$ for every $1 \le i \le r$. Write $S^{(i)}$ for those elements $s \in S \cap W_i$ that do not lie inside of $W_j$ for any $j<i$. This gives us a partition of $S$. Notice now that since each element of $T_V$ contains all elements of $S$ and $|T_V|=b'$ it follows that $S$ is $(k,b-b')$-free with respect to $T'=T \setminus T_V$. By induction, we then have that the number of incidences inside of $Z(P) \cap V$ is bounded by
 \begin{equation}
 \begin{aligned}
 &\sum_{i=1}^r \mathcal{I}(S^{(i)},T') \\
 \le & \, C \sum_{i=1}^r \tau_{d-1}(b,k)|S^{(i)}|^{\alpha_k(d-1)} \deg(T)^{\beta_k(d-1)} \deg(W_i)^{1-\alpha_k(d-1)} \\
 &+ (b-b'-1) \sum_{i=1}^r |S^{(i)}| + k \sum_{i=1}^r \deg(T) \deg(W_i) \\
 \le & \, C \tau_{d-1}(b,k)|S|^{\alpha_k(d-1)} \deg(T)^{\beta_k(d-1)} (M_{n-d} \deg(V))^{1-\alpha_k(d-1)} \\
  &+ (b-b'-1) |S| + C k \deg(T) \deg(V) M_{n-d} \\
 \le & \, C \tau_d(b,k)|S|^{\alpha_k(d)}\deg(T)^{\beta_k(d)} \deg(V)^{1-\alpha_k(d)} + (b-b'-1)|S|,
 \end{aligned}
 \end{equation}
where $C=O_n(1)$ is a constant whose precise value may change at each occurrence and where have used that, since $M_{n-d} \gtrsim_n 1$ and $d \ge 2$, we have
\begin{equation}
\begin{aligned}
C k \deg(T) \deg(V) M_{n-d} &\lesssim_n C k \deg(T) \deg(V) M_{n-d}^{d-1} \\
&\lesssim_n \tau_d(b,k)|S|^{\alpha_k(d)}\deg(T)^{\beta_k(d)} \deg(V)^{1-\alpha_k(d)}. \\
\end{aligned}
\end{equation}
Combining this with (\ref{500}) and the estimate $\mathcal{I}(S,T_V) \le b'|S|$, we get the desired result.

\subsection{Sharp constructions}
\label{83}

 As remarked in Section \ref{81}, Lemma \ref{trivialbound} is valid for an abstract set of elements $S$ and an arbitrary family $T$  of subsets of $S$, as long as $S$ is $(k,b)$-free with respect to $T$. This is a slightly weaker form of the Kövári-Sós-Turán bound \cite{KST}. On the other hand, constructing examples with the largest possible number of incidences in this abstract setting is a difficult task and is known as the Zarankiewicz problem \cite{Zz}. In particular, it is conjectured that for every choice of positive integers $b,k$ it is possible to find a set of elements $X$ and a family $Y$ of subsets of $X$, for general choices of sizes $|X|$ and $|Y|$, such that $X$ is $(k,b)$-free with respect to $Y$ and $\mathcal{I}(X,Y) \gtrsim_{b,k} |X||Y|^{1-1/k}$. This is currently known for $k \le 3$ \cite{Bo}. We also refer the reader to \cite{BK} for some recent progress. We now show a straightforward construction that would allow us to embed such examples into our algebraic setting, showing that the exponents and the dependency on the degrees in Theorem \ref{6} are tight in general. However, we remark that in some cases of Theorem \ref{6}, like the case of points and lines corresponding to the Szemerédi-Trotter theorem, lower bounds are known by different constructions (see also \cite{DS}).
 
 \begin{prop}
 Let $b,k$ and $1 \le d \le n$ be positive integers. Suppose there exists a set of elements $X$ and a family $Y$ of subsets of $X$ such that $X$ is $(k,b)$-free with respect to $Y$ and $\mathcal{I}(X,Y) \gtrsim_{b,k} |X||Y|^{1-1/k}$. Then, given any irreducible variety $V \subseteq \C^n$ of dimension $d$ with $V(\R)$ an irreducible real variety of dimension $d$ and with $\deg(V) \delta(V)^d \le c \frac{\mathcal{I}(X,Y)}{|Y|}$ for some sufficiently small $c>0$ depending on $n$ and $k$, we can find a family of hypersurfaces $T \subseteq \R^n$ and a set of points $S \subseteq V(\R)$ such that $S$ is $(k,b)$-free with respect to $T$ and 
 $$ \mathcal{I}(S,T) \gtrsim_{b,k,n} |S|^{\alpha_k(d)} \deg(T)^{\beta_k(d)} \deg(V)^{1-\alpha_k(d)}.$$
 \end{prop}

\begin{proof}
Our first step will be to modify $Y$ slightly so that each of its members is incident to the same number of elements of $X$ without affecting the $(k,b)$-free nature of the configuration. With this in mind, write $\sigma = \frac{\mathcal{I}(X,Y)}{|Y|} \gtrsim_{k,b} |X||Y|^{-1/k}$ for the average number of incidences of an element of $Y$. Clearly, if $Y_1 \subseteq Y$ consists of those elements contributing $< \sigma/2$ incidences and $Y_2 = Y \setminus Y_1$, we have that $\mathcal{I}(X,Y_2) \ge |Y| \sigma/2$. Since $|y| \ge \sigma/2$ for every $y \in Y_2$, we can associate to every such $y$ a family of disjoint subsets $a_1^y,\ldots, a_r^y \subseteq y$, each of size $\lfloor \sigma/4 \rfloor$ and with $\sum_{i=1}^r |a_i^y| \ge |y|/2$. Let $Y_3$ be the family of subsets $a_i^y$ of $X$ obtained in this way, with $y$ ranging over all elements of $Y_2$. Notice that choosing the constant $c$ in the statement sufficiently small with respect to $k$, we can guarantee that $\lfloor \sigma/4 \rfloor > k$. As a consequence, since $X$ is $(k,b)$-free with respect to $Y$, we see that each member of $Y_3$ can only have appeared as one of the subsets $a_i^y$ for at most $b$ choices of $y \in Y_2$. We conclude that
$$\mathcal{I}(X,Y_3) \ge \frac{1}{b} \sum_{y \in Y_2} |y|/2 = \frac{1}{2b} \mathcal{I}(X,Y_2) \ge |Y| \frac{\sigma}{4b} \gtrsim_{k,b} |X||Y|^{1-1/k},$$
 so if $|Y| \ge |Y_3|$, we have $\mathcal{I}(X,Y_3) \gtrsim_{k,b} |X||Y_3|^{1-1/k}$. On the other hand, if $|Y_3| > |Y|$, then by construction of the members of $Y_3$ we know that 
$$\mathcal{I}(X,Y_3) \ge |Y_3| \lfloor \sigma/4 \rfloor \gtrsim_{k,b} |X||Y_3|^{1-1/k}.$$
Now suppose $x_1,\ldots,x_k$ are elements of $X$ and $y_1,\ldots,y_b$ are elements of $Y_3$ such that $x_i \in y_j$ for all $1 \le i \le k, 1 \le j \le b$. This implies that the $y_j$ have non-empty intersection and so arise in our construction of $Y_3$ from subsets $a_i^y$ corresponding to different elements $y \in Y_2$. However, since $Y_2 \subseteq Y$, this contradicts that $X$ is $(k,b)$-free with respect to $Y$. We have thus shown that $X$ is $(k,b)$-free with respect to $Y_3$. Therefore, renaming $Y_3$ as $Y$, we see that we may assume that all elements of $Y$ are incident to the same number of elements of $X$. We write $K$ for this number of incidences, so $\mathcal{I}(X,Y) = K |Y|$.

Let now $V$ be as in the statement. In particular, since $K \gtrsim \sigma=\frac{\mathcal{I}(X,Y)}{|Y|}$, we have that $\deg(V) \delta(V)^d \le c' K$ for some $c'>0$ that we may take sufficiently small with respect to $n$. By deleting a few elements from each member of $Y$ as to reduce the value of $K$ slightly if necessary, we may then assume there is some $D$ such that 
$$\dim(\C[x_1,\ldots,x_n]_{\le D} / I(V)_{\le D}) = K+1.$$
Indeed, since the bound in Theorem \ref{8tool} is known to be an asymptotic identity \cite{Cha}, this can be accomplished reducing $K$ by at most an $O_n(1)$ constant and with $D \sim_n \left( \frac{K}{\deg(V)} \right)^{1/d}$. Write $X=\left\{ x_1,\ldots,x_{|X|} \right\}$. We define a set $S=\left\{ s_1,\ldots,s_{|S|} \right\} \subseteq V(\R)$ with $|S|=|X|$ by choosing its elements generically from $V(\R)$. Precisely, letting $g_1,\ldots,g_{K+1}$ be polynomials of degree at most $D$ projecting to a basis of $\C[x_1,\ldots,x_n]_{\le D} / I(V)_{\le D}$, we require that for all choices of $1 \le i_1 < \ldots < i_r \le |X|$ and $1 \le j_1 < \ldots < j_r \le K+1$, we have $\det( (g_{j_p}(s_{i_q}))_{1 \le p,q \le r} ) \neq 0$. Since this only requires us to recursively chose the elements outside of a proper algebraic subset of $V(\R)$, it can always be ensured. This guarantees that we cannot find a nontrivial linear combination of the $g_j$ vanishing on $K+1$ elements of $S$.

Each element $y \in Y$ contains elements $x_i \in X$ for every $i$ in a certain subset $I_y \subseteq \left\{ 1 ,\ldots, |X| \right\}$ with $|I_y| = K$. If we consider the corresponding subset of $S$ consisting of those $s_i$ with $i \in I_y$, since $|I_y| < K+1$, we know we can find a nontrivial linear combination $P_y$ of $g_1,\ldots,g_{K+1}$ vanishing on these elements. Notice that by the construction of $S$ given in the previous paragraph we know that $P_y$ does not vanish on any other element of $S$. This gives rise to a family of hypersurfaces $T$ of $\C^n$ of degree at most $D$, with $|T|=|Y|$ and so in particular $\deg(T) \le D|Y|$, and such that the incidence graph of $S$ and $T$ coincides with the incidence graph of $X$ and $Y$. We therefore conclude that $S$ is $(k,b)$-free with respect to $T$ and
\begin{equation}
\begin{aligned}
 \mathcal{I}(S,T) &\gtrsim_{b,k,n} |X||Y|^{1-1/k} \\
 &\gtrsim_{b,k,n} |X|^{\alpha_k(d)} |Y|^{\beta_k(d)} \left( |X| |Y|^{-1/k} \right)^{\beta_k(d)/d} \\
 &\gtrsim_{b,k,n}   |S|^{\alpha_k(d)}  (\frac{\deg(T)}{D})^{\beta_k(d)} K^{\beta_k(d)/d} \\
 &\gtrsim_{b,k,n} |S|^{\alpha_k(d)} \deg(T)^{\beta_k(d)} \deg(V)^{1-\alpha_k(d)}.
 \end{aligned}
 \end{equation}
 \end{proof}

\end{document}